\documentclass{amsart}
 
\usepackage{amsmath}
\usepackage{amsthm,amssymb,color,comment}
\usepackage{bbm}
\usepackage{hyperref}
\usepackage[TS1,T1]{fontenc}
\usepackage[latin1]{inputenc}
\usepackage{dsfont}
\usepackage{tikz}
\usepackage{enumitem}

\numberwithin{equation}{section}

\newtheorem{thm}{Theorem}[section]

\newtheorem{lem}[thm]{Lemma}
\newtheorem{cor}[thm]{Corollary}

\newtheorem{Def}[thm]{Definition}

\theoremstyle{definition}
\newtheorem{Ass}[thm]{Assumption}
\newtheorem{rem}[thm]{Remark}
\newtheorem{Ex}[thm]{Example}

\DeclareMathOperator*{\esssup}{ess\,sup}
\DeclareMathOperator*{\essinf}{ess\,inf}
\DeclareMathOperator{\DIV}{div}

\newcommand{\R}{\mathbb{R}}
\newcommand{\N}{\mathbb{N}}

\newcommand{\supp}{\text{supp}}
\newcommand{\diff}{\mathop{}\!\mathrm{d}}
\newcommand{\Clog}{C_{\scalebox{.7} {\mbox{log}}}}
\newcommand{\wstar}{\overset{\ast}{\rightharpoonup}}
\newcommand{\mus}{L_M(\Omega_T)}
\newcommand{\intmus}[1]{\int_{\Omega_T} M\left(t,x,#1\right) \diff t \diff x}
\newcommand{\intmusc}[1]{\int_{\Omega_T} M^*\left(t,x,#1\right) \diff t \diff x}
\newcommand{\armus}[1]{\xrightarrow{#1}}

\makeatletter
\newcommand{\doublewidetilde}[1]{{%
  \mathpalette\double@widetilde{#1}%
}}
\newcommand{\double@widetilde}[2]{%
  \sbox\z@{$\m@th#1\widetilde{#2}$}%
  \ht\z@=.9\ht\z@
  \widetilde{\box\z@}%
}
\makeatother
\author{Miroslav Bul\'i\v{c}ek}
\address{Mathematical Institute, Faculty of Mathematics and Physics, Charles University, Sokolovska 83, 186 75, Prague, Czech Republic}
\email{mbul8060@karlin.mff.cuni.cz}
\thanks{Miroslav Bul{\'\i}{\v{c}}ek was supported by the project No. 20-11027X financed by GA\v{C}R}

\author{Piotr Gwiazda}
\address{Institute of Mathematics of Polish Academy of Sciences, Jana i J\k edrzeja \'Sniadeckich 8, 00-656 Warsaw, Poland}
\email{pgwiazda@mimuw.edu.pl}
\thanks{Piotr Gwiazda was supported by National Science Center, Poland through project no. 2018/31/B/ST1/02289.}

\author{Jakub Skrzeczkowski}
\address{Faculty of Mathematics, Informatics and Mechanics, University of Warsaw, Stefana Banacha 2, 02-097 Warsaw, Poland}
\email{jakub.skrzeczkowski@student.uw.edu.pl}
\thanks{Jakub Skrzeczkowski was supported by National Science Center, Poland through project no. 2018/31/B/ST1/02289.}

\begin{document}

\title[Parabolic equations in Musielak - Orlicz spaces - discontinuous in time $N$-function]{Parabolic equations in Musielak - Orlicz spaces with discontinuous in time $N$-function}

\begin{abstract}
We consider a parabolic PDE with Dirichlet boundary condition and monotone operator $A$ with non-standard growth controlled by an $N$-function depending on time and spatial variable. We do not assume continuity in time for the $N$-function. Using an additional regularization effect coming from the equation, we establish the existence of weak solutions and in the particular case of isotropic $N$-function, we also prove their uniqueness. This general result applies to equations studied in the literature like $p(t,x)$-Laplacian and double-phase problems. 
\end{abstract}


\maketitle

%\tableofcontents
\section{Introduction}
\subsection{PDEs in Musielak - Orlicz spaces} This paper focuses on study of parabolic equations having the form
\begin{equation}\label{intro:parabolic_eq}
u_t(t,x) = \mbox{div} A(t,x,\nabla u(t,x)) + f(t,x) \mbox{ in } (0,T) \times \Omega,
\end{equation}
completed by the homogeneous Dirichlet boundary condition and the initial value $u_0(x)$. Here, $\Omega \subset \mathbb{R}^d$ is a bounded domain, $T$ denotes the length of time interval, $f: (0,T) \times \Omega \to \mathbb{R}$ is a measurable bounded function and $A$ is a monotone operator with coercivity and growth controlled by a so - called $N$-function $M: (0,T) \times \Omega \times \mathbb{R}^d \to \mathbb{R}$ (see Definition \ref{intro:def_Nfunc}), i.e. for almost all $(t,x) \in (0,T)\times \Omega$ and all $\xi \in \mathbb{R}^d$, we have:
\begin{equation}\label{intro:ass_on_A_eq}
M(t,x,\xi) + M^*(t,x,A(t,x,\xi)) \leq c\, A(t,x,\xi)\cdot \xi + h(t,x)
\end{equation}
where $M^*$ denotes the convex conjugate to $M$ (see Definition \ref{intro:def_compl_func}) and $h \in L^1((0,T)\times \Omega)$. Originally, problem \eqref{intro:parabolic_eq} was solved with $M(t,x,\xi) = |\xi|^p$ where $1 < p < \infty$. In this classical setting, \eqref{intro:ass_on_A_eq} implies that $A$, understood as a map
$$
L^p(0,T; W_0^{1,p}(\Omega)) \ni u \mapsto A(t,x,\nabla u) \in \left(L^p\left(0,T; W_0^{1,p}(\Omega)\right) \right)^*,
$$
is a bounded continuous operator and standard approaches (Galerkin method and compactness in Sobolev-Bochner spaces) applies (see \cite{brezis1979strongly, landes1981existence} and references therein) showing that the Sobolev space is an appropriate functional setting for problem \eqref{intro:parabolic_eq}. However, if the $N$-function $M$ appearing in \eqref{intro:ass_on_A_eq} has not a polynomial growth with respect to $\xi$ and is $(t,x)$-dependent, one has to look for a solution $u$ such that its gradient $\nabla u$ belongs to the Musielak - Orlicz space $L_M((0,T) \times \Omega)$, i.e. the space of measureable functions $\xi: (0,T) \times \Omega \to \mathbb{R}^d$ which satisfy
$$
\int_{ (0,T) \times \Omega } M\left(t,x,\frac{\xi(t,x)}{\lambda}\right) \diff t \diff x < \infty
$$
for some $\lambda > 0$, see Definition \ref{intro:def_mus_orlicz_space}. First results in this direction were focused on function $M$ being independent of $(t,x)$ and direction of $\xi$, i.e. $M(t,x,\xi) = N(|\xi|)$. Under the additional coercivity estimate $t^2 \ll N(t)$ and the so-called $\Delta_2$ condition for convex conjugate, i.e.
\begin{equation}\label{intro:Delta2_cond}
N^*(2t) \leq k N^*(t)
\end{equation}
for some constant $k$, this case was treated in \cite{donaldson1974inhomogeneous, robert1974inequations}. Another approach, introduced in \cite{elmahi2002strongly}, assumed growth bound $N(t) \ll t^{d/(d-1)}$ and condition
$
N(C t s) \leq N(t)N(s)
$ to be satisfied by $N$. Briefly speaking, condition \eqref{intro:Delta2_cond} provides a characterization of appropriate dual spaces (see \cite[Theorem 8.20]{adams2003sobolev}) and allows to extract weakly-$^*$ converging subsequences from bounded sequences. Similar methods have been used to study existence of solutions to \eqref{intro:parabolic_eq} with data ``below the duality'', i.e. $f \in L^1((0,T) \times \Omega)$, see \cite{gwiazda2015renormalized}. \\

\noindent Another approach is based on looking for hypothesis on $M$ implying that $C_0^{\infty}((0,T) \times \Omega)$ is a dense subset of $L_M((0,T) \times \Omega)$ (at least in the sense of modular convergence, see Definition \ref{intro:def_mod_conv}) so that one can test \eqref{intro:parabolic_eq} with the solution itself. It is a classical fact that for variable Lebesgue spaces (i.e. $M(t,x,\xi) = |\xi|^{p(t,x)}$) some continuity of $p$ in $(t,x)$ is in general necessary (see \cite[Example 6.12]{cruz2013variable}). Density argument was first exploited to establish well-posedness of \eqref{intro:parabolic_eq} for $M(t,x,\xi) = N(|\xi|)$ in \cite{elmahi2005parabolic} and it was extended later to cover more and more general functions $M$ without assumption of the form \eqref{intro:Delta2_cond} but with some sort of continuity hypothesis with respect to $(t,x)$ \cite{chlebicka2018well, chlebicka2019parabolic, gwiazda2011parabolic, MR3164940,swierczewska2014nonlinear} with the most general condition given in \cite{chlebicka2019parabolic}. We remark that similar progress have been made for elliptic equations and we refer the reader to the excellent review \cite{chlebicka2018pocket} discussing PDEs in Musielak - Orlicz spaces in detail.\\

\noindent We want to emphasize here that all papers mentioned above have a disadvantage on the continuity assumption of $N$-function $M(t,x,\xi)$ with respect to $t$. However, this cannot be optimal. One can consider the PDE of the form:
$$
u_t = \begin{cases}
\mbox{div} \nabla u & \mbox{ in }(0,1] \times \Omega, \\
\mbox{div} \left(|\nabla u|^{2} \nabla u \right)& \mbox{ in } (1,2]\times \Omega, 
\end{cases}
$$
which can be solved piecewisely (first on time interval $(0,1]$ and then on $(1,2]$) so one can develop well-posedness theory. We remark that in the recent monograph \cite[Section 2.2]{antontsev2015evolution} there is an example of degenerated parabolic equation
\begin{equation}\label{intro:pareq_with_wpt}
u_t - \mbox{div}(|u|^{\gamma(t,x)} \nabla u) = f,
\end{equation}
where the exponent $\gamma(t,x)$ satisfies bounds $-1 < \gamma_{-} \leq \gamma(t,x) \leq \gamma_{+} < \infty$ and $\nabla \gamma \in L^2((0,T)\times \Omega)$. Then, \eqref{intro:pareq_with_wpt} has at least one bounded weak solution. Moreover, if
$$
\gamma_{-} > 0 \qquad \mbox{ and } \qquad \esssup_{x \in \overline{\Omega}} |\nabla\gamma(t,x)| \in L^2(0,T),
$$
the solution becomes unique. However, these results are strongly based on the particular form of the operator in \eqref{intro:pareq_with_wpt}. Finally, let us remark that many problems that are of current interests can be studied in the framework  \\

\noindent In this paper we establish the existence of solutions to \eqref{intro:parabolic_eq} in the Musielak - Orlicz space $L_M((0,T)\times \Omega)$ without any assumption on continuity of $M(t,x,\xi)$ with respect to $t$ (see Theorem \ref{mainresult:existence}). Moreover, for isotropic $N$-functions of the form $M(t,x,|\xi|)$ we obtain the uniqueness in a given class \footnote{Note that in case of spatially boundary conditions, we have the uniqueness of a weak solution even without any structural assumption on $M$.}. The main features of our work are:
\begin{itemize}
\item In contrast to works described above, we do not try to approximate {\it every} function in modular topology but only the distributional solution to \eqref{intro:parabolic_eq}. Using the equation satisfied by the solution, we can retrieve the missing regularity in time and proceed without continuity with respect to time assumption for $M(t,x,\xi)$. Similar approaches have been used for renormalized solutions to the transport equation, see \cite[Section 2.1]{de2007ordinary}.
\item Existence result is deduced by using only the local versions of standard methods: the energy equality \eqref{res:local_energy_equality} and the monotonicity method in Section \ref{sect:monotonicity_trick}.
\item Uniqueness result is based on the global energy equality \eqref{res:global_energy_equality_without_truncation} that can be deduced from the local one.
\end{itemize}
We remark that \eqref{intro:parabolic_eq} generalizes a great variety of parabolic problems and we refer to \cite[Corollary 1.1-1.2, Example 1.1-1.2]{chlebicka2019parabolic} for a long list of examples with assosciated $N$-functions. This includes double phase problems where the $N$-function $M$ is trapped between two power-type functions. Such equations have been studied by Marcellini \cite{MR969900,MR1094446} and they are still subject of active research \cite{MR3570955,MR3775180,MR3294408}.\\

\noindent Finally, we would like to emphasize that the studied problem has not only a theoretical background but can find an application in physically well-motivated problems whenever one considers rapid changes of the underlying equations with respect to time variable. As a prototypic example may serve the flow of incompressible electrorheological fluids (see \cite{diening2011lebesgue} or \cite{ruzicka2000electrorheological} for more details). These fluids are described by the system of equations:
\begin{equation*}
\begin{split}
\DIV {\bf v} &=0,\\
\partial_t {\bf v} + \DIV( {\bf v} \otimes {\bf v}) - \DIV {\bf S} &= - \nabla p + {\bf g} + \nabla {\bf E} \cdot {\bf P},
\end{split}
\end{equation*}
where ${\bf v} = (v_1, v_2, v_3)$ denotes the velocity of the fluid, ${\bf S}$ is the viscous stress tensor, ${\bf E}$ is the electrical intensity and ${\bf P}$ is the polarization. Note that in the case of no electric field present we have 
$$
{\bf S} \sim D({\bf v}) \mbox{ where } D({\bf v}) = \frac{1}{2}\left(\nabla {\bf v} + \left(\nabla {\bf v} \right)^{T} \right).
$$
But, when we apply an electric field, the viscous stress changes dramatically and behaves like ${\bf S} \sim \left|D({\bf v})\right|^{r(t,x)} D({\bf v})$ with some function $r(t,x)$. Hence, it is evident that we are now in the case corresponding to the choice of $N$-function $M(t,x,\xi) = |\xi|^{r(t,x)}$, where $r(t,x)$ is discontinuous with respect to time variable.
\subsection{Musielak - Orlicz spaces}
\noindent In this subsection we briefly recall theory of Musielak - Orlicz spaces. For detailed discussion, we refer the reader to the classical book \cite{musielak2006orlicz} as well as to a modern presentation \cite{chlebicka2019book} aimed at applications in PDEs.\\

\noindent In what follows, $\Omega \subset \mathbb{R}^d$ denotes a bounded domain and $T > 0$ is arbitrary. We set $\Omega_T := (0,T)\times \Omega$. 
\begin{Def}[Young function]\label{intro:def_Nfunc_iso}
We say that $m: [0,\infty) \to [0,\infty)$ is a Young function if the following holds true:
\begin{enumerate}[label=(Y\arabic*)]
\item $m(s) = 0 \iff s = 0$,
\item $m$ is convex,
\item \label{intro:def_Nfunc_iso_prop3} $m$ is superlinear, i.e. $\lim_{s\to 0} \frac{m(s)}{s} = 0$ and $\lim_{s\to \infty} \frac{m(s)}{s} = \infty$. \label{intro:def_hom_iso_superlinearity}
\end{enumerate}
\end{Def}
\begin{Def}[$N$-function]\label{intro:def_Nfunc}
We say that $M: \Omega_T \times \mathbb{R}^d \to \mathbb{R}$ is $N$-function if the following holds true:
\begin{enumerate}[label=(M\arabic*)]
\item{$M(t,x,\xi) = M(t,x,-\xi)$ for a.e. $(t,x) \in \Omega_T$ and all $\xi \in \mathbb{R}^d$,}
\item{$M(t,x,\xi)$ is a Carath\'eodory function, i.e. for a.e. $(t,x) \in \Omega_T$, the mapping $\mathbb{R}^d \ni \xi \mapsto M(t,x,\xi)$ is continuous and for all $\xi \in \mathbb{R}^d$, the mapping $\Omega_T \ni (t,x) \mapsto M(t,x,\xi)$ is measurable,}
\item{for a.e. $(t,x) \in \Omega_T$, the map $\mathbb{R}^d \ni \xi \mapsto M(t,x,\xi)$ is convex,}
\item{there exist two Young functions $m_1$, $m_2$ such that for almost all $(t,x) \in \Omega_T$ and all $\xi \in \mathbb{R}^d$ we have
$$
m_1(|\xi|) \leq M(t,x,\xi) \leq m_2(|\xi|).
$$ \label{intro:defNf_control_isotropic}}
\end{enumerate}
\end{Def}
%intro:defNf_control_isotropic intro:def_Nfunc
\begin{Def}[Convex conjugate]\label{intro:def_compl_func}
Let $m$ be a Young function. Then, we define its convex conjugate $m^*$ as
$$
m^*(s) = \sup_{t\in [0,\infty)} (st - m(t)).
$$
Similarly, if $M$ is an $N$-function, we define its convex conjugate $M^*$ as
$$
M^*(t,x,\eta) = \sup_{\xi\in \mathbb{R}^d} (\xi \cdot \eta - M(t,x,\xi)).
$$
\end{Def}
\begin{lem}[Properties of $N$-functions]\label{intro:lem_prop_Nfunc}
Let $m$ be a Young function and $M$ be an $N$-function. Then:
\begin{enumerate}[label=(N\arabic*)]
\item \label{intro:propNfunct:item1} function $\frac{m(t)}{t}$ is nondecreasing,
\item \label{intro:propNfunct:item2} $m^*$ is a Young function,
\item \label{intro:propNfunct:item3} $M^*$ is an $N$-function,
\item \label{intro:propNfunct:item4} $\lim_{|\xi|\to 0} \esssup_{(t,x) \in \Omega_T} \frac{M(t,x,\xi)}{|\xi|} = 0$ and $\lim_{|\xi|\to \infty} \essinf_{(t,x) \in \Omega_T} \frac{M(t,x,\xi)}{|\xi|} = \infty$,
\item \label{intro:propNfunct:item5} if $f_n:\Omega_T \to \mathbb{R}^d$ is a sequence of functions and $\int_{\Omega_T} M(t,x,f_n(t,x)) \diff t \diff x \leq C$ independently of $n$, then $\{f_n\}_{n \in \mathbb{N}}$ is equi-integrable,
\item  \label{intro:propNfunct:item55} if $f_n:\Omega_T \to \mathbb{R}^d$ is a sequence of functions and $\int_{\Omega_T} M(t,x,f_n(t,x)) \diff t \diff x \leq C$ for some $C>1$ then $\|f_n \|_{L_M} \leq C$,
\item \label{intro:propNfunct:item6} if $f_n:\Omega_T \to \mathbb{R}^d$ is a sequence of functions such that $f_n \to f$ a.e. in $\Omega_T$ and $\| f_n \|_{\infty} \leq C$ independently of $n$, then $\int_{\Omega_T} M(t,x,f_n(t,x)) \diff t \diff x \to \int_{\Omega_T} M(t,x,f(t,x)) \diff t \diff x$.
\end{enumerate}
\end{lem}
\begin{proof}
Let $t \leq s$. By convexity of $m$, we have
$$
\frac{m(t)}{t} =\frac{1}{t} m\left(\frac{t}{s} s + \left(1- \frac{t}{s}\right) 0 \right) \leq 
\frac{1}{t}\frac{t}{s} m(s) = \frac{m(s)}{s},
$$
which proves \ref{intro:propNfunct:item1}. \\

\noindent To see property \ref{intro:propNfunct:item2}, we observe directly from Definition \ref{intro:def_compl_func} that $m^*(0) = 0$ as $m \geq 0$ and $m(0) = 0$. The convexity of $m^*$ follows as it is a supremum of affine maps. Hence, it remains to check \ref{intro:def_Nfunc_iso_prop3} in Definition \ref{intro:def_Nfunc_iso}. For any $\lambda > 0$
$$
\liminf_{s \to \infty} \frac{m^*(s)}{s} \geq \frac{\lambda s - m(\lambda)}{s} \geq  \lambda
$$
which proves $\lim_{s \to \infty} \frac{m^*(s)}{s} = \infty$. Now, let $\delta > 0$ and $s \in (0,\delta)$ be arbitrary. Then,
$$
\frac{m^*(s)}{s} =
\sup_{t \in [0,\infty)} \left(t - \frac{m(t)}{s}\right)
 = \sup_{t \in [0,\infty)} t \left(1 - \frac{m(t)}{t} \frac{1}{s} \right)
\leq \sup_{t \in [0,\infty)} t \left(1 - \frac{m(t)}{t} \frac{1}{\delta} \right)
$$
However, for $t$ such that $\frac{m(t)}{t} \geq \delta$, the maximized expression is negative. By property \ref{intro:propNfunct:item1} and \ref{intro:def_Nfunc_iso_prop3} in Definition \ref{intro:def_Nfunc_iso}, we find $t_{\delta}$, such that $\frac{m(t_{\delta})}{t_{\delta}} = \delta$ and we get that
$$
\frac{m^*(s)}{s} \leq \sup_{t \in [0,t_{\delta}]} t \left(1 - \frac{m(t)}{t} \frac{1}{\delta} \right) \leq t_{\delta}.
$$ 
We claim that $t_{\delta} \to 0$ as $\delta \to 0$. For if not, $C_2 \geq t_{\delta} \geq C_1 > 0$ for some constants $C_1$ and $C_2$. But then
$$
\delta = \frac{m(t_{\delta})}{t_{\delta}} \geq \frac{m(C_1)}{C_2} > \frac{m(0)}{C_2} = 0,
$$
since $m$ is strictly increasing and $m(0) = 0$. This proves \ref{intro:propNfunct:item2}. To see \ref{intro:propNfunct:item3}, we observe that 
$$
m_1(|\xi|) \leq M(t,x,\xi) \leq m_2(\xi) \implies
m_2^*(|\xi|) \leq M^*(t,x,\xi) \leq m_1^*(\xi).
$$
Since $m_1^*$ and $m_2^*$ are Young functions, the conclusion follows. Property \ref{intro:propNfunct:item4} is a consequence of \ref{intro:defNf_control_isotropic} in Definition \ref{intro:def_Nfunc} and superlinearity of Young functions \ref{intro:def_hom_iso_superlinearity}. To deduce \ref{intro:propNfunct:item5}, we note that 
$$
\int_{\Omega_T} m_1(|f_n(t,x)|) \diff t \diff x \leq C
$$
and it is well-known that such bound for superlinear function $m_1$ is equivalent to uniform integrability on bounded domains, see \cite[Proposition 1.27]{Ambrosio2000}. Property \ref{intro:propNfunct:item55} follows by convexity:
$$
\intmus{\frac{f_n(t,x)}{C}} \leq \frac{1}{C} \intmus{f_n(t,x)} \leq 1.
$$
Finally, as Young function are increasing, property \ref{intro:propNfunct:item6} follows by Dominated Convergence Theorem.
\end{proof}
\begin{rem}
In previous works on PDEs in Musielak - Orlicz spaces, $N$-functions were defined slightly differently using combination of conditions in Definition \ref{intro:def_Nfunc_iso}, Definition \ref{intro:def_Nfunc} and Lemma \ref{intro:lem_prop_Nfunc} (see, for instance, \cite{bulivcek2019existence, chlebicka2018well, chlebicka2019parabolic}). We believe that Definition \ref{intro:def_Nfunc} makes our work more accessible for readers not familiar with this setting.
\end{rem}
\begin{Def}[Musielak - Orlicz space $L_M(\Omega_T)$]\label{intro:def_mus_orlicz_space}
Let $M$ be an $N$ - function. Then, the Musielak - Orlicz space $L_M(\Omega_T)$ is defined as
$$
L_M(\Omega_T) = \left\{ \xi: \Omega_T \to \mathbb{R}^d: \mbox{ there is } \lambda>0 \mbox{ such that } \int_{\Omega_T} M\left(t,x, \frac{\xi(t,x)}{\lambda}\right) \diff t \diff x < \infty \right\}.
$$
This is a Banach space equipped with the norm
\begin{equation}\label{intro:musor_norm}
\| \xi \|_{L_M} = \inf \left\{\lambda>0: \int_{\Omega_T} M\left(t,x, \frac{\xi(t,x)}{\lambda}\right) \diff t \diff x \leq 1 \right\}.
\end{equation}
If $m$ is a Young function, we can similarly define the Musielak - Orlicz space $L_m(\Omega_T)$.
\end{Def}
\noindent The following form of the Young and the H\"older inequalities are true in Musielak-Orlicz spaces (see \cite[Lemma 2.4]{wroblewska2010steady}):
\begin{lem}\label{intro:Hold_Mus_inequality}
Let $M$ be an $N$-function and $M^*$ be its convex conjugate. Then, for all $\xi \in \mus$ and $\eta \in L_{M^*}(\Omega_T)$:
\begin{enumerate}[label=(I\arabic*)]
\item \label{intro:Hold_Mus_inequality_Y} $\int_{\Omega_T} \xi(t,x) \eta(t,x) \diff t \diff x \leq \intmus{\xi(t,x)} + \intmusc{\eta(t,x)}$,
\item $\int_{\Omega_T} \xi(t,x) \eta(t,x) \diff t \diff x \leq 2 \| \xi \|_{L_M} \|\eta\|_{L_{M^*}}.$
\end{enumerate}
\end{lem} 
\noindent As convergence in norm in space $\mus$ seems to be too strong for applications in PDEs, we introduce the concept of modular convergence.
\begin{Def}[Modular convergence in $\mus$]\label{intro:def_mod_conv}
We say that sequence of functions $\{\xi_n\}_{n \in \mathbb{N}} \subset \mus$ converges to $\xi$ modularly if there exists $\lambda > 0$ such that
$$
\intmus{\frac{\xi_n(t,x) - \xi(t,x)}{\lambda}} \to 0.
$$
We write $\xi_n \armus{M} \xi$. By convexity, if follows that if $\{\xi_n\}_{n \in \N} \subset \mus$ and $\xi_n \armus{M} \xi$ then $\xi \in \mus$.
\end{Def}
\noindent Note that modularly converging sequences converge in $L^1(\Omega_T)$ and so, they have a subsequence converging a.e. As in the case of classical Lebesgue spaces, simple functions are dense in $\mus$ with respect to the modular convergence:
\begin{lem}[Density of simple functions]\label{intro:dens_simpl_fcn}
Let $\xi \in \mus$. Then, there is a sequence $\{\xi_n\}_{n \in \N}$ of simple functions such that $\xi_n \armus{M} \xi$. 
\end{lem}
\noindent Due to Vitali Convergence Theorem (cf. \cite[Exercise 15, Section 6.1]{folland2013real}), we have the following characterization of modular convergence and its corollary.
\begin{thm}\label{intro:vitali_for_mo}
Let $\{\xi_n \}_{n \in \N} \subset \mus$ and $\xi \in \mus$. Then, $\xi_n \armus{M} \xi$ if and only if the following hold:
\begin{enumerate}[label=(V\arabic*)]
\item $\{\xi_n \}_{n \in \N}$ converges to $\xi$ in measure, 
\item $\left\{M\left(t,x,\frac{\xi_n}{\lambda}\right) \right\}_{n \in \N}$ is uniformly equi-integrable for some $\lambda > 0$.
\end{enumerate}
\end{thm}
\begin{cor}\label{intro:hold_conv_inmod}
Let $\{\varphi_j \}_{j \in \N} \subset \mus$ and $\{\phi_j\}_{j\in \N} \subset L_{M^*}(\Omega_T)$. Suppose that $\varphi_j \armus{M} \varphi$ and $\phi_j \armus{M^*} \phi$. Then, $\varphi_j \, \phi_j \to \varphi \, \phi$ in $L^1(\Omega_T)$.
\end{cor}
\begin{proof}
By Theorem \ref{intro:vitali_for_mo}, $\varphi_j \to \varphi$ and $\phi_j \to \phi$ in measure, and so $\varphi_j \cdot \phi_j \to \varphi \cdot \phi$ also in measure. To conclude, we have to prove uniform integrability of $\{\varphi_j \cdot \phi_j \}$. However, by Young's inequality, for any $Q \subset \Omega_T$:
\begin{equation}\label{Young_applied_to_prove_UI}
\int_Q \frac{\varphi_j(t,x) \cdot \phi_j(t,x)}{\lambda} \diff t \diff x \leq
\int_Q M\left(t,x, \frac{\varphi_j(t,x)}{\lambda}\right) \diff t \diff x
+
\int_Q M^*\left(t,x, \frac{\phi_j(t,x)}{\lambda}\right) \diff t \diff x.
\end{equation}
Again, Theorem \ref{intro:vitali_for_mo} implies existence of $\lambda_1, \lambda_2 > 0$ such that sequences $\left\{M\left(t,x, \frac{\varphi_j(x)}{\lambda_1}\right)  \right\}$ and $\left\{M^*\left(t,x, \frac{\phi_j(x)}{\lambda_2}\right)  \right\}$ are uniformly integrable. Taking $\lambda = \max(\lambda_1, \lambda_2)$ in \eqref{Young_applied_to_prove_UI}, we conclude the proof. 
\end{proof}
\noindent Finally, we discuss some compactness results allowing to extract converging subsequences.
\begin{Def}[Subspace $E_M(\Omega_T)$]
$E_M(\Omega_T)$ is a closure of bounded functions in the norm \eqref{intro:musor_norm}.
\end{Def}
\noindent It is easy to see by approximation with simple functions that $E_M(\Omega_T)$ is separable. Therefore, \cite[Theorem 2.6]{wroblewska2010steady} and the Banach-Alaoglu-Bourbaki Theorem (cf. \cite[Theorem 3.16 and Corollary 3.30]{brezis2010functional}) yields:
\begin{lem}\label{intro:weak_compactness}
We have the following duality characterization $\left(E_M(\Omega_T)\right)^* = L_{M^*}(\Omega_T)$. In particular, if $\{\xi_n\}_{n \in N}$ is a bounded sequence in $L_{M^*}(\Omega_T)$, it has a weakly-$\ast$ converging subsequence.
\end{lem}
\noindent For Young functions, we also define Orlicz--Sobolev spaces and we recall their basic properties (cf. \cite[Chapter 8]{adams2003sobolev}). 
\begin{Def}[Orlicz--Sobolev space]\label{intro:Orlicz-Sob_Def}
Let $m: \R \to \R$ be a Young function. We define Orlicz--Sobolev spaces $W^1_0L_m(\Omega_T)$ as
$$
W^1_0L_m(\Omega_T) = \left\{ \xi \in L^1(0,T;W^{1,1}_0(\Omega)):
\| \xi \|_{L_m}, \| \nabla \xi \|_{L_m} < \infty \right\}
$$
and equip it with the norm
$$
\| \xi \|_{W^1L_m} = \| \xi \|_{L_m} + \| \nabla \xi \|_{L_m}.
$$
We also consider its subset $W^1_0E_m(\Omega_T)$:
$$
W^1_0E_m(\Omega_T) = \left\{ \xi \in W^1_0L_m:
\xi \in E_m(\Omega_T) \mbox{ and } \nabla \xi \in E_m(\Omega_T) \right\}
$$
\end{Def}
\begin{lem}[Properties of $W^1_0E_m(\Omega_T)$ and $W^1_0L_m(\Omega_T)$]\label{intro:properties_orlicz_sob_lemma}
Spaces $W^1_0E_m(\Omega_T)$ and $W^1_0L_m(\Omega_T)$ have the following properties:
\begin{enumerate}[label=(P\arabic*)]
\item $W^1_0E_m(\Omega_T)$ is separable,
\item \label{prop_lemma:density} space $C_0^{\infty}((0,T) \times \Omega)$ is dense in $W^1_0E_m(\Omega_T)$ with respect to $\|\cdot\|_{L_m}$ norm,
\item \label{Poinc_ineq} (Poincar\'e inequality, cf. \cite[Corollary 4.1]{chlebicka2019elliptic}) there are constants $c_1$ and $c_2$ such that for all $u \in W^1_0L_m(\Omega_T)$, 
$$
\int_{\Omega_T} m(c_1|u|) \diff t \diff x \leq c_2 \int_{\Omega_T} m(|\nabla u|) \diff t \diff x. 
$$
In particular, $\| \nabla u \|_{L_m}$ is an equivalent norm on $W^1_0L_m(\Omega_T)$.
\end{enumerate}
\end{lem}
\subsection{Main result} We start with assumptions on $\mathcal{N}$-function $M$ and operator $A$.\\

\begin{Ass}[Assumptions on $M$]\label{intro:ass_on_M} We assume that $M:\Omega_T  \times \R^d \to \mathbb{R}$ is an $N$-function. Moreover, we assume that there is a function {$\Theta:(0,T) \times [0,1] \times [0,\infty) \to [0,\infty)$}, which is nondecreasing with respect to the second and the third variable, such that
$$
\forall{{ C>1}} ~ \forall{\delta_0 > 0} ~\exists{R > 0} \mbox{ such that for a.e. $t \in (0,T)$ and all $\delta \leq \delta_0$ there holds } \Theta(t,\delta, C \delta^{-1}) \leq {R}.
$$
This function describes relation between $M(t,x,\xi)$ and $M_Q(t,\xi) = \essinf_{x \in \Omega \cap 5{Q_{\phantom{i}}}} M(t,x,\xi)$, where $Q \subset \mathbb{R}^d$ is an arbitrary cube and $5{Q_{\phantom{i}}}$ is a cube with the same center as $Q$ with five times longer edge. More precisely, we assume that there exists $\xi_0 \in \R$ and $\delta_0 > 0$ such that for every cube $Q \subset \mathbb{R}^d$ with edge $\delta \in (0, \delta_0)$ and all $\xi \in \R^d$ with $|\xi|>\xi_0$ we have
\begin{equation}\label{intro:our_ass_on_M}
\frac{M(t,x,\xi)}{M_{Q}^{**}(t,\xi)} \leq \Theta(t,\delta,|\xi|),
\end{equation}
where $M_{Q}^{**}$ is the second convex conjugate to $M_Q$, see Definition \ref{intro:def_compl_func}.
\end{Ass}
{ \noindent We remark that Assumption \ref{intro:ass_on_M} mimics the one made in \cite{chlebicka2019parabolic}, namely
$$
\frac{M(t,x,\xi)}{M_{Q,I}^{**}(\xi)} \leq \Theta(\delta,|\xi|),
$$
where $M_{Q,I}(\xi) = \essinf_{x \in \Omega \cap 3Q, t \in I \cap (0,T)} M(t,x,\xi)$, $Q$ is a cube with edge of length $\delta$, $I$ is a subinterval of $\R$ with $|I| \leq \delta$ and function $\Theta$ satisfies:
$$
\forall{C>0} ~ \forall{\delta_0 > 0} ~\exists{R > 0} \mbox{ such that for a.e. $t \in (0,T)$ and all $\delta \leq \delta_0$ there holds } \Theta(\delta, C \delta^{-d}) \leq {R}.
$$
We also note that our assumption is equivalent to condition (A1') in \cite[Definition 4.1.1]{MR3931352} while the latter to the condition (A1-n) in \cite[Chapter 7.3]{MR3931352}.\\

\noindent On the one hand, the relaxed regularity in time allows for $N$-functions which are merely measurable in time. On the other hand, we need to control the quantity $\Theta(t, \delta, C\,\delta^{-1})$ rather than $\Theta(\delta, C\, \delta^{-d})$ which results in better exponents regimes for some well-known examples of $N$-functions, see Example \ref{ex:Nfunctions}. This improvement is based on the observation that in the approximation result one needs to approximate in the modular topology functions of the form 
$$
\nabla(T_k(u) + \varphi) \mbox{ where } \nabla u \in L_M(\Omega_T), \, \varphi \in C_c^{\infty}(\Omega_T)
$$
The observation described above can be easily implemented in the previous works on this topic cf. \cite{chlebicka2018well,chlebicka2019parabolic}.\\
}
\begin{rem}\label{rem:cond:ass_isotropic_case}
In the particular case of an isotropic $N$-function $M(t,x,|\xi|)$, Assumption \ref{intro:ass_on_M} boils down to existence of the function {$\Theta: (0,T) \times [0,1] \times [0,\infty) \to [0,\infty)$} which is nondecreasing with respect to second and third variable such that
\begin{equation}\label{cond:ass_isotropic_case}
\limsup_{\delta \to 0^+} \Theta(t, \delta, C\delta^{-1}) \mbox{ is bounded uniformly in time } t \in (0,T)
\end{equation}
and
$$
\frac{M(t,x,r)}{M(t,y,r)} \leq \Theta(t,|x-y|,r). 
$$
See \cite[Lemma A.4]{chlebicka2019parabolic} for the proof.
\end{rem}
{ 
\begin{Ex}\label{ex:Nfunctions} We list here $N$-functions satisfying Assumptions \ref{intro:ass_on_M}. For the proof, we refer to Appendix \ref{app:example_Nfunctions}.
\begin{enumerate}[label=(E\arabic*)]
\item \label{nfunc_ex_var_exp} $M(t,x,\xi) = |\xi|^{p(t,x)}$ with $1<p_{-} \leq p(t,x) \leq p_{+} < \infty$ and $p(t,x) \in L^{\infty}(0,T; \Clog(\Omega))$. Here, $\Clog(\Omega)$ is the space of log-H\"older continuous functions on $\Omega$, i.e. functions $v: \Omega \to \R$ such that
$$
|v(x) - v(y)| \leq -\frac{C}{\log|x-y|}
$$
for all $x, y \in \Omega$ and some constant $C$. Note that only very low regularity of $p(t,x)$ in time is required.
\item $M(t,x,\xi) = |\xi|^{p(t,x)} + a(t,x)\, |\xi|^{q(t,x)}$ where
\begin{itemize}
\item $1 < p_{-} \leq p(t,x) < p^+ <\infty$, $1 < q_{-} \leq q(t,x) < q^+ <\infty$,
\item $p(t,x), q(t,x) \in L^{\infty}(0,T; \Clog(\Omega))$, \item $a(t,x) \in L^{\infty}(0,T; C^{\alpha}(\Omega))$ for some $\alpha \in (0,1)$ and $a \geq 0$,
\item
$
q(t,x) - p(t,x) \leq \alpha.
$
\end{itemize}
Here, $C^{\alpha}(\Omega)$ is the space of $\alpha$-H\"older continuous functions on $\Omega$. We stress that only very low regularity of $p(t,x)$ and $q(t,x)$ in time is required. We also observe that for $p_{-} < d$, our admissible regime of exponents is better than $q(t,x) - p(t,x) \leq \frac{\alpha\,p_{-}}{d}$ known from \cite{chlebicka2019parabolic}.
\end{enumerate}
\end{Ex}
}
\begin{Ass}[Assumptions on $A$]\label{intro:ass_on_A} We assume that $A:\Omega_T  \times \R^d \to \mathbb{R}^d$ satisfies:
\begin{enumerate}[label=(A\arabic*)]
\item \label{intro:ass_on_A:continuity}$A$ is a Carath\'eodory's function, i.e. for a.e. $(t,x) \in \Omega_T$, map $\mathbb{R}^d \ni \xi \mapsto A(t,x,\xi)$ is continuous and for all $\xi \in \mathbb{R}^d$, map $\Omega_T \ni (t,x) \mapsto A(t,x,\xi)$ is measurable,
\item \label{intro:ass_on_A:coercgr}(coercivity and growth bound) there is a constant $c$ and function $h \in L^{\infty}(\Omega_T)$ such that for all $\xi \in \R^d$ and a.e. $(t,x) \in \Omega_T$:
$$
M(t,x,\xi) + M^*(t,x,A(t,x,\xi)) \leq c \, A(t,x,\xi)\cdot \xi + h(t,x),
$$
\item \label{intro:assumA_mono}(monotonicity) for all $\eta,\xi \in \R^d$ and a.e. $(t,x) \in \Omega_T$:
$$
(A(t,x,\xi) - A(t,x,\eta)) \cdot (\xi - \eta) \geq 0,
$$
\item  \label{intro:assumA_vanish} for a.e. $(t,x) \in \Omega_T$ we have $A(t,x,0) = 0$.
\end{enumerate} 
\end{Ass}
\begin{rem}
In classical papers, condition \ref{intro:assumA_vanish} could be deduced from coercivity and growth bounds. Here, \ref{intro:ass_on_A:coercgr} implies only that
$$
0 \leq M^*(t,x,A(t,x,0)) \leq h(t,x).
$$
We believe that \ref{intro:assumA_vanish} can be waived. Nevertheless, we make this assumption as it is natural and it simplifies many technical computations.
\end{rem}
{
\begin{Ex}\label{ex:operatorsA} We list here functions $\mathcal{A}$ corresponding to $N$-functions in Example \ref{ex:Nfunctions} which satisfy Assumptions \ref{intro:ass_on_A}. For the proof, we refer to Appendix \ref{app:example_A}.
\begin{enumerate}[label=(F\arabic*)]
\item\label{exA:ptx} $A(t,x,\xi) = |\xi|^{p(t,x)-2} \xi$ leads to the equation with $p(t,x)$-Laplacian
$$
u_t(t,x) = \DIV \left[|\nabla u(t,x)|^{p(t,x)-2}\,\nabla u(t,x)\right] + f(t,x)
$$
and the governing $N$-function $M(t,x,\xi)$ is given by \ref{nfunc_ex_var_exp} in Example \ref{ex:Nfunctions}. Such problems have been considered recently for instance in \cite{MR4074602,MR4139121} under assumption that $p(t,x)$ is log-H\"older continuous jointly in $t$ and $x$. In our setting, we only need $p(t,x) \in L^{\infty}(0,T;\Clog(\Omega))$.
\item\label{exA:double_phase} $A(t,x,\xi) = |\xi|^{p(t,x)-2}\,\xi + a(t,x)\, \,|\xi|^{q(t,x)-2}\,\xi$ leads to the double phase problem
$$
u_t(t,x) = \DIV \left[|\nabla u(t,x)|^{p(t,x)-2}\,\nabla u(t,x) + a(t,x)\,|\nabla u(t,x)|^{q(t,x)-2}\,\nabla u(t,x) \right] + f(t,x).
$$
Such problems were studied with variational methods \cite{MR3073153,MR4074614} but mostly with constant or only $x$-dependent exponents. The case of $p(t,x)$ and $q(t,x)$ which are log-H\"older continuous jointly in $t$ and $x$ was studied in \cite{chlebicka2019parabolic}. Our theory requires only $p(t,x), q(t,x) \in L^{\infty}(0,T;\Clog(\Omega))$.
\end{enumerate}
\end{Ex}
}
\begin{lem}\label{intro:lem_bound_on_A}
Let $A$ satisfy Assumption \ref{intro:ass_on_A}. Then, for every $K>0$, there exists a constant $C(K)$ depending on $K$ such that $|A(t,x,\xi)| \leq C(K)$ for a.e. $(t,x) \in \Omega_T$ and all $\xi \in \mathbb{R}^d$ fulfilling $|\xi| \leq K$.
\end{lem}
\begin{proof}
Let $|\xi| \leq K$. Assumption \ref{intro:ass_on_A:coercgr} implies that
\begin{equation}\label{intro:Aisbdd_estimate}
M^*(t,x,A(t,x,\xi)) \leq c \, A(t,x,\xi)\cdot \xi + h(t,x).
\end{equation}
Let $m$ be a Young function such that $m(|\xi|) \leq M^*(t,x,\xi)$ for a.e. $(t,x) \in \Omega_T$ as in point \ref{intro:defNf_control_isotropic} in Definition \ref{intro:def_Nfunc}. If $|A(t,x,\xi)| \leq 1$, the assertion follows by choosing $C(K) \geq 1$. Otherwise, \eqref{intro:Aisbdd_estimate} implies
$$
\frac{m(|A(t,x,\xi)|)}{|A(t,x,\xi)|} \leq c \, |\xi| + \|h\|_{\infty} \leq c\,K + \|h\|_{\infty}.
$$
Since map $s \mapsto \frac{m(s)}{s}$ is nondecreasing (property \ref{intro:propNfunct:item1} in Lemma \ref{intro:lem_prop_Nfunc}) and $m$ is superlinear (property \ref{intro:def_Nfunc_iso_prop3} in Definition \ref{intro:def_Nfunc_iso}), the assertion follows.
\end{proof} 

\noindent Next, we define a function space relevant for the problem \eqref{intro:parabolic_eq} as follows:
$$
V^M_T = \left\{u: \Omega_T \to \R \mbox{ such that } u \in L^1(0,T; W^{1,1}_0(\Omega)), \nabla u \in \mus \mbox{ and } u \in L^{\infty}(0,T; L^2(\Omega))\right\}.
$$
The main results of this paper read:
\begin{thm}[Existence of solutions]\label{mainresult:existence}
Suppose that Assumptions \ref{intro:ass_on_M} and \ref{intro:ass_on_A} are satisfied. Let $\Omega \subset \R^d$ be a bounded Lipschitz domain, $u_0 \in L^{\infty}(\Omega)$ and $f \in L^{\infty}(\Omega)$. Then, there exists $u \in V^M_T(\Omega)$ which is a weak solution to \eqref{intro:parabolic_eq}. More precisely, there exists $u \in V^M_T(\Omega)$ such that $A(t,x,\nabla u) \in L_{M^*}(\Omega_T)$ and for all $\varphi \in C^{\infty}_0([0,T) \times \Omega)$, there holds:
\begin{equation*}
\begin{split}
&-\int_{\Omega_T} u(t,x) \partial_t \varphi(t,x) \diff t \diff x 
- \int_{\Omega} u_0(x)\varphi(0,x) \diff x + \\ 
&\qquad \qquad \qquad \qquad \qquad+
\int_{\Omega_T} A(t,x,\nabla u) \cdot \nabla \varphi(t,x) \diff t \diff x = 
\int_{\Omega_T} f(t,x) \varphi(t,x) \diff t \diff x.
\end{split}
\end{equation*}
In addition, $u$ satisfies the global energy inequality, i.e. for all $t \in [0,T]$ there holds
\begin{equation}\label{theorem:energyinequality}
\frac{1}{2}\int_{\Omega} \left[u^2(t,x) - u_0^2(x) \right] \diff x \leq 
- \int_0^t \int_{\Omega} A(s,x, \nabla u(s,x)) \cdot \nabla u(s,x) \, \diff x \diff s + \int_0^t \int_{\Omega} f(s,x) \,u(s,x) \, \diff x \diff s.
\end{equation}
\end{thm}
%%%%%%
\begin{thm}[Uniqueness of solutions]\label{mainresult:uniqueness}
Let all assumptions of Theorem \ref{mainresult:existence} be satisfied. Moreover, suppose that the $N$-function $M$ is isotropic, i.e. it is of the form $M(t,x,|\xi|)$. Then, weak solution to \eqref{intro:parabolic_eq} is unique and it satisfies the energy equality, i.e. for all $t \in [0,T]$ there holds
\begin{equation}\label{theorem:energyequality}
\frac{1}{2} \int_{\Omega} \left[u^2(t,x) - u_0^2(x) \right] \diff x =
- \int_0^t \int_{\Omega} A(s,x, \nabla u(s,x)) \cdot \nabla u(s,x) \, \diff x \diff s + \int_0^t \int_{\Omega} f(s,x) \,u(s,x) \, \diff x \diff s.
\end{equation}
\end{thm}
%%%%%%%%
\section{Auxillary theory and results}
\subsection{Smooth approximation}\label{res:section_on_mollification}
\noindent In this section we prove that if $u \in V^M_T(\Omega)$, then $u$ can be approximated in the modular topology of the gradients. We formulate this result locally in $\Omega$ but we remark that the similar approach has already been used in \cite[Theorem 3.1]{chlebicka2019parabolic}, where approximation was performed globally for Lipschitz domains $\Omega$ by using a decomposition on star-shaped sets, see \cite[Lemma II.1.3]{galdi2011introduction}.\\

\noindent First, we recall the definition of truncation and mollification operators:
\begin{Def}
[Truncation]\label{res:trunc_def} Function
$$
T_k (s) = \begin{cases}
s & \mbox{if } |s| \leq k,\\
k \frac{s}{|s|} & otherwise,
\end{cases}
$$
is called truncation at level $k$. We also denote by $G_k$ its primitive function, i.e. we set
$$
G_k(s) = \int_0^s T_k(\sigma) \diff \sigma.
$$
\end{Def}
\begin{Def}[Mollification with respect to the spatial variable]\label{res:mol_in_sp}
Let $\eta:\R^d \to \R$ be a standard regularizing kernel, i.e. $\eta$ is a smooth nonnegative function compactly supported in a ball of radius one and fulfills $\int_{\mathbb{R}^d} \eta(x) \diff x = 1$. Then, we set $\eta_{\varepsilon}(x) = \frac{1}{\varepsilon^d} \eta\left(\frac{x}{\varepsilon}\right)$ and for arbitrary $u: \Omega \to \R$ and $\Omega' \Subset \Omega$, we define $u^{\varepsilon}: \Omega' \to \R$ as
$$
u^{\varepsilon}(x) = \int_{\R^d} \eta_{\varepsilon}(x-y) u(y) \diff y.
$$
Furthermore, if $u: \Omega_T \to \mathbb{R}$, then $u^{\varepsilon}$ denotes mollification in space, i.e.
$$
u^{\varepsilon}(t,x) = \int_{\R^d} \eta_{\varepsilon}(x-y) u(t,y) \diff y.
$$
\end{Def}
\begin{Def}[Mollification with respect to time]\label{res:mol_in_ti}
Let $\zeta:\R \to \R$ be a standard regularizing kernel, i.e. $\zeta$ is a smooth nonnegative function compactly supported in a ball of radius one and fulfills $\int_{\mathbb{R}} \zeta(x) \diff x = 1$. Then, we set $\zeta_{\varepsilon}(x) = \frac{1}{\varepsilon} \zeta\left(\frac{x}{\varepsilon}\right)$ and for arbitrary $u: \R \times \Omega \to \R$, we define $S^{\varepsilon}u: \R \times \Omega \to \R$ as
$$
S^{\varepsilon}u(t,x) = \int_{\R} \zeta_{\varepsilon}(t-s) u(s,x) \diff s.
$$
\end{Def}
\noindent For properties of mollified functions, the reader may consult \cite[Appendix C.4]{evans1998partial}. Finally, we formulate the approximative properties of the mollifications defined above, which is the most essential tool used in the paper.
\begin{thm}\label{res:approx_theorem}
Let $\Omega \subset \R^d$, $\psi:\Omega \to \R$ be compactly supported satisfying $0 \leq \psi \leq 1$ and $u \in V^M_T(\Omega)$. Suppose that Assumption \ref{intro:ass_on_M} is satisfied. Then, there exists $\varepsilon_0 > 0$:
\begin{enumerate}[label=(S\arabic*)]
\item $\left(T_k(u^{\varepsilon}) \psi\right)^{\varepsilon} \in L^1(0,T;C_0^{\infty}(\Omega))$ for all $\varepsilon \in (0, \varepsilon_0)$,
\item $T_k(u^{\varepsilon}) \psi \to T_k(u) \psi$ a.e. in $\Omega_T$ and in $L^1(0,T; L^1(\Omega))$ as $\varepsilon \to 0^+$,
\item $\nabla\left(T_k(u^{\varepsilon}) \psi\right)^{\varepsilon} \armus{M} \nabla\left(T_k(u) \psi\right)$ as $\varepsilon \to 0^+$, where the modular convergence $\armus{M}$ is defined in Definition \ref{intro:def_mod_conv}. 
\end{enumerate}
\end{thm}
\noindent The key estimate needed for the proof of Theorem \ref{res:approx_theorem} is formulated in the following lemma.
\begin{lem}\label{res:approx_thm_helpf_lem}
Suppose that Assumption \ref{intro:ass_on_M} is satisfied, $v:\Omega_T \to \R^d$ and $v \in \mus$ with $\intmus{v(t,x)} < \infty$. Assume that $v = \nabla u + \varphi$ for some $u \in V^M_T(\Omega)$ and $\varphi \in L^{\infty}(\Omega_T)$. Then, there is a constant $C$ such that for any compactly supported $\psi:\Omega \to \R$ with $0 \leq \psi \leq 1$ and for all $k\in \mathbb{N}$,
\begin{multline*}
\limsup_{\varepsilon \to 0} \intmus{ \big(\mathds{1}_{\left|u^{\varepsilon}\right| \leq k}\,v^{\varepsilon}(t,x) \psi(x)\big)^{\varepsilon}} \leq \\ \leq
\int_{\Omega_T} m_2 \left(|v(t,x)|\psi(x) \right) \mathds{1}_{|v(t,x)|\psi(x) \leq \xi_0} \diff t \diff x +
C \intmus{v(t,x)},
\end{multline*}
where $\xi_0$ is a constant from Assumption \ref{intro:ass_on_M} and $m_2$ is a Young function as in \ref{intro:defNf_control_isotropic} in Definition~\ref{intro:def_Nfunc}.
\end{lem}
\begin{rem}
Since $v \in \mus$, the condition $\intmus{v(t,x)} < \infty$ can be always satisfied by considering appropriate scaling if necessary.
\end{rem}
\begin{proof}[Proof of Lemma \ref{res:approx_thm_helpf_lem}]
To shorten all formulas, we denote $z_{\varepsilon}(t,x) = \big(\mathds{1}_{\left|u^{\varepsilon}\right| \leq k} \, v^{\varepsilon}(t,x) \psi(x)\big)^{\varepsilon}$ and write:
\begin{equation}\label{res:proof_of_approx_thm_splitstep}
\begin{split}
\intmus{\big(\mathds{1}_{\left|u^{\varepsilon}\right| \leq k} \,v^{\varepsilon}(t,x) \psi(x)\big)^{\varepsilon}} \leq &\\ \leq
\int_{\Omega_T} M\left(t,x,z_{\varepsilon}(t,x)\right) \mathds{1}_{|z_{\varepsilon}(t,x)| \leq \xi_0} \diff t \diff x &+
\int_{\Omega_T} M\left(t,x,z_{\varepsilon}(t,x)\right) \, \mathds{1}_{|z_{\varepsilon}(t,x)| > \xi_0} \diff t \diff x.
\end{split}
\end{equation}
For the first term, we use \ref{intro:defNf_control_isotropic} in Definition \ref{intro:def_Nfunc} to observe:
$$
\int_{\Omega_T} M\left(t,x,z_{\varepsilon}(t,x)\right) \mathds{1}_{|z_{\varepsilon}(t,x)| \leq \xi_0} \diff t \diff x \leq
\int_{\Omega_T} m_2\left(t,x,z_{\varepsilon}(t,x)\right) \mathds{1}_{|z_{\varepsilon}(t,x)| \leq \xi_0} \diff t \diff x 
$$
and so, by \ref{intro:propNfunct:item6} in Lemma \ref{intro:lem_prop_Nfunc} we get
\begin{equation}\label{res:proof_of_reg_easy_part}
\limsup_{\varepsilon \to 0} \int_{\Omega_T} M\left(t,x,z_{\varepsilon}(t,x)\right) \mathds{1}_{|z_{\varepsilon}(t,x)| \leq \xi_0} \leq
\int_{\Omega_T} m_2 \left(|v(t,x)|\psi(x) \right) \mathds{1}_{|v(t,x)|\psi(x) \leq \xi_0} \diff t \diff x.
\end{equation}
Hence, it is sufficient to focus on the second term in \eqref{res:proof_of_approx_thm_splitstep}. Let $\{Q_j\}_{j=1}^{N_{\varepsilon}}$ be a family of closed cubes with edge $\varepsilon$ such that $\mbox{int} Q_j \cap \mbox{int} Q_i = \emptyset$ for $i \neq j$ and $\Omega \subset \cup_{i=1}^{N_{\varepsilon}} Q_i$. Moreover, let $3{Q_i}$ and $5{Q_i}$ be the cubes with the same center as $Q_i$ and edges $3\varepsilon$ and $5\varepsilon$, respectively. Then,
\begin{multline*}
\int_{\Omega_T} M\left(t,x,z_{\varepsilon}(t,x)\right) \mathds{1}_{|z_{\varepsilon}(t,x)| > \xi_0}  \diff t \diff x = \\=\sum_{i=1}^{N_{\varepsilon}}
\int_0^T  \int_{Q_i \cap \Omega} 
\frac{M\left(t,x,z_{\varepsilon}(t,x)\right)}{M_{Q_i}^{**}(t, z_{\varepsilon}(t,x))} M_{Q_i}^{**}(t, z_{\varepsilon}(t,x))  \mathds{1}_{|z_{\varepsilon}(t,x)| > \xi_0} \diff x \diff t,
\end{multline*}
where $M_{Q_i}^{**}$ is defined in Assumption \ref{intro:ass_on_M}. Note that we assume that $v = \nabla u + \varphi$ for some $u \in V^M_T(\Omega)$ and $\varphi \in L^{\infty}(\Omega_T)$. We note that
$$
z_{\varepsilon}(t,x) =  \big(\nabla T_k(u^{\varepsilon}(t,x)) \psi(x)\big)^{\varepsilon} + 
 \big(\mathds{1}_{\left|u^{\varepsilon}\right| \leq k} \, \varphi^{\varepsilon}(t,x) \psi(x)\big)^{\varepsilon} := z_{\varepsilon}^1(t,x) + z_{\varepsilon}^2(t,x).
$$
Clearly, using Young's convolutional inequality, we have $\left| z_{\varepsilon}^2(t,x)\right| \leq \| \varphi \|_{\infty} \, \| \psi \|_{\infty}$. Moreover,
$$
z_{\varepsilon}^1(t,x) = -\,\big(T_k(u^{\varepsilon}) \, \DIV \psi \big) \ast \eta_{\varepsilon} (t,x)+
\big(T_k(u^{\varepsilon}) \psi\big) \ast \nabla \eta_{\varepsilon} (t,x)
$$
so applying Young's convolutional inequality we have:
\begin{equation*}
|z_{\varepsilon}^1(t,x)| \leq k \, \|\DIV \psi\|_{\infty} + \frac{k \, \| \psi \|_{\infty}\, \| \nabla \eta_{\varepsilon}\|_{1}}{\varepsilon}.
\end{equation*}
We conclude that $|z_{\varepsilon}(t,x)| \leq \frac{C(k,\varphi, \eta)}{\varepsilon}$ for $\varepsilon < 1$ and therefore, using \eqref{intro:our_ass_on_M}, we get that for $x \in Q_i \cap \Omega$ the following inequality
$$
\frac{M\left(t,x,z_{\varepsilon}(t,x)\right)}{M_{Q_i}^{**}(t, z_{\varepsilon}(t,x))}
\leq \Theta\left (t, \delta,  \frac{C(k,\varphi, \eta)}{\varepsilon} \right) \leq C
$$
holds true for sufficiently small $\varepsilon$. Consequently,
\begin{equation}\label{res:proof_of_reg_splitting}
\int_{\Omega_T} M\left(t,x,z_{\varepsilon}(t,x)\right) \mathds{1}_{|z_{\varepsilon}(t,x)| > \xi_0} \diff t \diff x \leq
C \sum_{i=1}^{N_{\varepsilon}} \int_0^T  \int_{Q_i \cap \Omega} 
 M_{Q_i}^{**}(t, z_{\varepsilon}(t,x)) \diff x \diff t.
\end{equation}
To estimate the right hand side in the above inequality, we focus on each summand separately. Using Jensen's and Young's convolutional inequalities we deduce:
\begin{equation}\label{res:estimate_using_convex_env}
\begin{split}
& \int_0^T  \int_{Q_i \cap \Omega} 
 M_{Q_i}^{**}(t, z_{\varepsilon}(t,x)) \diff x \diff t  \\
& \qquad = \int_0^T  \int_{Q_i \cap \Omega} 
 M_{Q_i}^{**}\left(t, \int_{B(0,\varepsilon)} \eta_{\varepsilon}(y) \Big(v^{\varepsilon}(t,x-y) \psi(x-y) \mathds{1}_{|u^{\varepsilon}|\leq k}(x-y) \Big) \diff y \right) \diff x \diff t 
 \\
 & \qquad  \leq 
  \int_0^T  \int_{Q_i \cap \Omega} \int_{B(0,\varepsilon)}  \eta_{\varepsilon}(y)
 M_{Q_i}^{**}\big(t,  v^{\varepsilon}(t,x-y) \psi(x-y) \mathds{1}_{|u^{\varepsilon}|\leq k}(x-y)  \big) \diff y \diff x \diff t 
  \\
 & \qquad  \leq
   \int_0^T  \int_{\R^d} \int_{B(0,\varepsilon)}  \eta_{\varepsilon}(y)
 M_{Q_i}^{**}\big(t,  v^{\varepsilon}(t,x-y) \psi(x-y) \mathds{1}_{3{Q_i} \cap \Omega}(x-y) \big) \diff y \diff x \diff t 
  \\
 &
 \qquad  \leq 
   \int_0^T  \int_{\R^d}  M_{Q_i}^{**}\big(t,  v^{\varepsilon}(t,x) \psi(x) \mathds{1}_{3{Q_i} \cap \Omega}(x) \big)  \diff x \diff t = \int_0^T  \int_{3{Q_i} \cap \Omega}  M_{Q_i}^{**}\big(t,  v^{\varepsilon}(t,x) \psi(x) \big)  \diff x \diff t, 
\end{split}
\end{equation}
where we used the fact that $\|\eta_{\varepsilon}\|_{L^1} = 1$ and the fact that $M_{Q_i}^{**}(t,\xi) = 0 \iff \xi = 0$. Next, by convexity of $\xi \mapsto M_{Q_i}^{**}(t,\xi)$ and thanks to $0 \leq \psi(x) \leq 1$, we can simply estimate the last term as
$$
\int_0^T  \int_{3{Q_i} \cap \Omega}  M_{Q_i}^{**}\big(t,  v^{\varepsilon}(t,x) \psi(x) \big)  \diff x \diff t \leq
\int_0^T  \int_{3{Q_i} \cap \Omega \cap \supp(\psi)}   M_{Q_i}^{**}\big(t,  v^{\varepsilon}(t,x) \big)  \diff x \diff t.
$$
Then, repeating the procedure from \eqref{res:estimate_using_convex_env}, we deduce
$$
\int_0^T  \int_{3{Q_i} \cap \Omega \cap \supp(\psi)}   M_{Q_i}^{**}\big(t,  v^{\varepsilon}(t,x) \big)  \diff x \diff t \leq
\int_0^T  \int_{5{Q_i} \cap \Omega \cap \supp(\psi)}   M_{Q_i}^{**}\big(t,  v(t,x) \big)  \diff x \diff t.
$$
Finally, as $M_{Q_i}(t,\xi) = \mbox{ess inf}_{x \in \Omega \cap 5{Q_{i}}} M(t,x,\xi)$ and since $M^{**}_{Q_i}(t,\xi) \leq M_{Q_i}(t,\xi)$, we can estimate each summand by the above inequality to get:
$$
\int_0^T  \int_{5{Q_i} \cap \Omega \cap \supp(\psi)}   M_{Q_i}^{**}\big(t,  v(t,x) \big)  \diff x \diff t \leq 
\int_0^T  \int_{5{Q_i} \cap \Omega} M(t,x, v(t,x)) \diff x \diff t.
$$
Coming back to \eqref{res:proof_of_reg_splitting}, we obtain
\begin{equation}\label{res:proof_of_reg_almost_there!}
\int_{\Omega_T} M\left(t,x,z_{\varepsilon}(t,x)\right) \mathds{1}_{|z_{\varepsilon}(t,x)| > \xi_0} \diff t \diff x \leq
\int_{\Omega_T} M\left(t,x,v(t,x)\right) \diff t \diff x
\end{equation}
for some possibly different constant $C$ which can be increased due to integration over repeating parts of overlaping cubes $\left\{5{Q_i} \right\}_{i=1}^{N_{\varepsilon}}$. Combining \eqref{res:proof_of_reg_easy_part} with \eqref{res:proof_of_reg_almost_there!}, we finish the proof. 
\end{proof}
\begin{proof}[Proof of Theorem \ref{res:approx_theorem}]
First two properties follow from properties of mollification and continuity of the truncation. To show also the third property, we first compute:
$$
\nabla\left(T_k(u^{\varepsilon}) \psi\right)^{\varepsilon} = 
\left(\mathds{1}_{|u^{\varepsilon}| \leq k} (\nabla u)^{\varepsilon} \psi\right)^{\varepsilon} + 
\left(T_k(u^{\varepsilon}) \nabla\psi\right)^{\varepsilon}.
$$ 
Then, due to \ref{intro:propNfunct:item6} in Lemma \ref{intro:lem_prop_Nfunc},
$
\left(T_k(u^{\varepsilon}) \nabla\psi\right)^{\varepsilon} \armus{M} 
T_k(u) \nabla\psi
$
and so, it is sufficient to focus only on the first term. Using Lemma \ref{intro:dens_simpl_fcn}, we find a sequence of simple functions $\{\varphi_n\}_{n \in \N}$ such that $\varphi_n \to \nabla u$ a.e. and $\varphi_n \armus{M} \nabla u$ as $n \to \infty$, i.e. there is $\widetilde{\lambda} > 0$ such that
$$
\intmus{\frac{\nabla u(t,x) - \varphi_n(t,x) }{\widetilde{\lambda}}} \to 0.
$$
Then, for some $\lambda_1, \lambda_2, \lambda_3$ to be chosen later, $\lambda = \lambda_1 + \lambda_2 + \lambda_3$ and some $n \in \mathbb{N}$ we write:
\begin{equation*}
\begin{split}
&\intmus{\frac{\left(\mathds{1}_{|u^{\varepsilon}| \leq k} (\nabla u)^{\varepsilon} \psi\right)^{\varepsilon} - \mathds{1}_{|u| \leq k} \nabla u \psi}{\lambda}} \leq  \\
& \qquad \qquad \leq \frac{\lambda_1}{\lambda} \intmus{\frac{\left(\mathds{1}_{|u^{\varepsilon}| \leq k} (\nabla u)^{\varepsilon} \psi\right)^{\varepsilon} - \left(\mathds{1}_{|u^{\varepsilon}| \leq k} (\varphi_n)^{\varepsilon} \psi\right)^{\varepsilon}}{\lambda_1}} + \\
& \qquad \qquad + \frac{\lambda_2}{\lambda} \intmus{\frac{\left(\mathds{1}_{|u^{\varepsilon}| \leq k} (\varphi_n)^{\varepsilon} \psi\right)^{\varepsilon} - \mathds{1}_{|u| \leq k} \varphi_n \psi }{\lambda_2}} \\
& \qquad \qquad + \frac{\lambda_3}{\lambda} \intmus{\mathds{1}_{|u| \leq k} \psi \frac{ \varphi_n  -  \nabla u}{\lambda_3}} =: A^{n,\varepsilon} + B^{n,\varepsilon} + C^{n,\varepsilon}.
\end{split}
\end{equation*}
Using \ref{intro:propNfunct:item6} in Lemma \ref{intro:lem_prop_Nfunc}, for any $n \in \N$ and $\lambda_2 > 0$, $\limsup_{\varepsilon \to 0} B^{n,\varepsilon} = 0$. Also, we note that
\begin{multline*}
\frac{\lambda_1}{\lambda} \intmus{\frac{\left(\mathds{1}_{|u^{\varepsilon}| \leq k} (\nabla u)^{\varepsilon} \psi\right)^{\varepsilon} - \left(\mathds{1}_{|u^{\varepsilon}| \leq k} (\varphi_n)^{\varepsilon} \psi\right)^{\varepsilon}}{\lambda_1}} \leq \\ \leq
\frac{\lambda_1}{\lambda} \intmus{\frac{\left(\mathds{1}_{|u^{\varepsilon}| \leq k} (\nabla u - \varphi_n)^{\varepsilon} \psi\right)^{\varepsilon}}{\lambda_1}}.
\end{multline*}
Therefore, if we choose $\lambda_1 = \lambda_3 = \widetilde{\lambda}$ and use Lemma \ref{res:approx_thm_helpf_lem}, we obtain
\begin{multline*}
\limsup_{\varepsilon \to 0} \left(A^{n,\varepsilon} + C^{n,\varepsilon} \right) \leq \\ \leq
\intmus{\mathds{1}_{|u| \leq k} \psi \frac{ \varphi_n  -  \nabla u}{\widetilde{\lambda}}} + \int_{\Omega_T} m_2 \left(\left|\frac{ \varphi_n  -  \nabla u}{\widetilde{\lambda}}\right|\psi(x) \right) \mathds{1}_{\left|\frac{ \varphi_n  -  \nabla u}{\widetilde{\lambda}}\right|\psi(x) \leq \xi_0} \diff t \diff x.
\end{multline*}
Since $\varphi_n \to \nabla u$ a.e. in $\Omega_T$ and $\varphi_n \armus{M} \nabla u$, we conclude the proof.
\end{proof}
\subsection{Regularization of the operator}
\noindent In this section, we formulate well-posedness theory for parabolic equations in Musielak-Orlicz spaces with Young functions. This allows us to construct solution to our problem by a limiting procedure. The following result was proven by Elmahi and Meskine \cite[Theorem 2]{elmahi2005parabolic} using Galerkin's approximation and mollification as in Section \ref{res:section_on_mollification} (however here $N$-function is homogeneous and isotropic so the result can be established significantly easier).
\begin{thm}\label{res:iso_wpt_thm}
Let $\Omega \subset \R^d$ be a bounded domain with segment property. Let $m: \R \to \R$ be a Young function. Suppose that $a: \Omega_T \times \R^d \to \R^d$ satisfies:
\begin{enumerate}[label=(R\arabic*)]
\item \label{reg_thm:ass1}$a$ is a Carath\'eodory's function, i.e. for a.e. $(t,x) \in \Omega_T$, map $\mathbb{R}^d \ni \xi \mapsto a(t,x,\xi)$ is continuous and for all $\xi \in \mathbb{R}^d$, map $\Omega_T \ni (t,x) \mapsto a(t,x,\xi)$ is measurable,
\item \label{reg_thm:ass2} there are $c \in E_{m^*}(\Omega_T)$ with $c \geq 0$ and nonnegative constant $\beta$ and $\gamma$ such that
$$
|a(t,x,\xi)| \leq \beta \Big( c(t,x) + (m^*)^{-1}(m(\gamma |\xi|)) \Big),
$$  
\item \label{reg_thm:ass3} there are $d \in L^1(\Omega_T)$ and nonnegative constants $\alpha$ and $\lambda$ such that
$$
a(t,x,\xi)\cdot \xi + d(t,x) \geq \alpha \, m\left(\frac{|\xi|}{\lambda} \right),
$$
\item \label{reg_thm:ass4} $a$ is stronly monotone, i.e. for all $\eta,\xi \in \R^d$ and a.e. $(t,x) \in \Omega_T$:
$$
(a(t,x,\xi) - a(t,x,\eta)) \cdot (\xi - \eta) > 0.
$$
\end{enumerate}
Then, the problem 
$$
 u_t = \DIV a(t,x,\nabla u) + g
$$
with $u(0,x) = u_0(x) \in L^{\infty}(\Omega_T)$, $u(t,x) = 0$ for $x \in \partial \Omega$ and $g \in  L^{\infty}(\Omega_T)$ has the unique weak solution $u \in C((0,T); L^2(\Omega)) \cap W^1L_{m}(\Omega_T)$ (see Definition \ref{intro:Orlicz-Sob_Def}).
\end{thm}
\noindent Using Theorem \ref{res:iso_wpt_thm}, one can define a sequence approximating solutions to \eqref{intro:parabolic_eq} as follows:
\begin{lem}\label{res:approximation_theta_lem}
Suppose $A$ satisfies Assumption \ref{intro:ass_on_A}, $M$ is an $N$-function and $m$ is a Young function such that $M(t,x,\xi) \leq m(|\xi|)$. For $\theta \in (0,1]$, consider regularized operator
\begin{equation}\label{res:regularized_problem_operator}
A_{\theta}(t,x,\xi) = A(t,x,\xi) + \theta \nabla_{\xi} m(|\xi|).
\end{equation}
Then, there exists a weak solution to the problem
\begin{equation}\label{res:regularized_problem}
 u^{\theta}_t = \DIV A_{\theta}(t,x,\nabla u^{\theta}) + g
\end{equation}
with $u(0,x) = u_0(x) \in L^{\infty}(\Omega_T)$, $u(t,x) = 0$ for $x \in \partial \Omega$ and $g \in  L^{\infty}(\Omega_T)$. More precisely, 
$$
u^{\theta} \in C((0,T);L^2(\Omega)) \cap L^1(0,T; W^{1,1}_0(\Omega)).
$$
Moreover, $u^{\theta}$ satisfies the global energy equality:
\begin{equation}\label{global_energy_regular_problem}
\begin{split}
&\int_{\Omega} \big[(u^{\theta}(t,x))^2 - (u_0(x))^2 \big] \diff x =
\\ &  \qquad \qquad \qquad 
-\, \int_0^t \int_{\Omega} A^{\theta}(s,x, \nabla u^{\theta}(s,x)) \cdot \nabla u^{\theta}(s,x) \, \diff x \diff s + \int_0^t \int_{\Omega} g(s,x) \,u^{\theta}(s,x) \, \diff x \diff s.
\end{split}
\end{equation}
We also have bounds which are uniform in $\theta$:
\begin{enumerate}[label=(C\arabic*)]
\item \label{unif_est_1} sequence $\{u^{\theta} \}_{\theta}$ is uniformly bounded in $L^{\infty}(0,T; L^2(\Omega))$,
\item \label{unif_est_2} sequence $\{\nabla u^{\theta} \}_{\theta}$ is uniformly bounded in $\mus$,
\item \label{unif_est_3} sequence $\{A(t,x,\nabla u^{\theta}) \}_{\theta}$ is uniformly bounded in $L_{M^*}(\Omega_T)$,
\item \label{unif_est_4} sequence $\{ \theta m^*(\nabla_{\xi} m(|\nabla u^{\theta}|)) \}_{\theta}$ is uniformly bounded in $L^1(\Omega_T)$.
\end{enumerate}
\end{lem}
\begin{proof}
First, we observe we observe from the definition of the convex conjugate that
\begin{equation}\label{useful_FYineq_eq}
\nabla_{\xi}m(|\xi|) \cdot \xi = m(|\xi|) + m^*(|\nabla_{\xi}m(|\xi|)|).
\end{equation}
We also note that $\nabla_{\xi}m(|\xi|) = m'(|\xi|)\frac{\xi}{|\xi|}$ so that $\nabla_{\xi}m(|\xi|)\xi \geq 0$. Let us check that assumptions of Theorem \ref{res:iso_wpt_thm} are satisfied with operator \eqref{res:regularized_problem_operator} controlled by $N$-function $m$. Assumption \ref{reg_thm:ass1} is fulfilled trivially. To verify \ref{reg_thm:ass2}, we use \eqref{useful_FYineq_eq}, \ref{intro:ass_on_A:coercgr} in Assumption \ref{intro:ass_on_A} and the convexity, to obtain:
\begin{equation}\label{check_ass_oneside}
\setlength{\jot}{6pt}
\begin{split}
c A_{\theta}(t,x,\xi) \cdot \xi &\geq M(t,x,\xi) + M^*(t,x,A(t,x,\xi)) - h(t,x) + c \, \theta \, \nabla_{\xi}m(|\xi|) \cdot \xi \\[2pt]
& \geq 0 + m^*(|A(t,x,\xi)|) - h(t,x) + c \, \theta \,m^*(|\nabla_{\xi}m(|\xi|)|) \\
& \geq 2 \min(1,c) \left( \frac{1}{2}m^*(|A(t,x,\xi)|) + \frac{1}{2} \,m^*(\theta |\nabla_{\xi}m(|\xi|)|) \right) - |h(t,x)|  \\[-1pt]
& \geq 2 \min(1,c) \, m^*\left(\frac{1}{2} \left|A_{\theta}(t,x,\xi)\right| \right) - |h(t,x)|.
\end{split}
\end{equation} 
On the other hand, by Young's inequality
\begin{equation}\label{check_ass_secside}
c A_{\theta}(t,x,\xi) \cdot \xi  \leq  \min(1,c) \, m\left(\frac{c}{\min^2(1,c)}|\xi|\right) + \min(1,c) \, m^*\left(\frac{1}{2} \left|A_{\theta}(t,x,\xi)\right|\right).
\end{equation}
Hence, we combine \eqref{check_ass_oneside} and \eqref{check_ass_secside} to deduce
$$
\min(1,c) \, m^*\left(\frac{1}{2} \left|A_{\theta}(t,x,\xi)\right| \right) \leq \min(1,c) \, m\left(\frac{c}{\min^2(1,c)}|\xi|\right) + |h(t,x)|.
$$
Next, we abbreviate $c_1 = 1/\min(1,c)$ and $c_2 = \frac{c}{\min^2(1,c)}$. Furthermore, since $m^*$ is increasing and convex, then $(m^*)^{-1}$ is increasing and concave. Moreover $(m^*)^{-1}(0) = 0$ so $(m^*)^{-1}$ is subadditive and therefore
$$
\frac{1}{2} \left|A_{\theta}(t,x,\xi)\right|  \leq (m^*)^{-1} \Big( m\left(|\xi|\right) + c_1|h(t,x)|\Big) \leq 
(m^*)^{-1} \big( m\left(|\xi|\right) \big) + (m^*)^{-1} \big(c_1|h(t,x)|\big),
$$
which proves \ref{reg_thm:ass2} since $h \in L^{\infty}(\Omega_T)$. Then, repeating computation in \eqref{check_ass_oneside} and applying \eqref{useful_FYineq_eq} we deduce:
\begin{equation}\label{check_third_ass}
\begin{split}
c A_{\theta}(t,x,\xi) \cdot \xi &\geq M(t,x,\xi) + M^*(t,x,A(t,x,\xi)) - h(t,x) + c \, \theta \, \nabla_{\xi}m(|\xi|) \cdot \xi \\[3pt]
& \geq c \, \theta \, m(|\xi|) - h(t,x),
\end{split}
\end{equation} 
which proves \ref{reg_thm:ass3}. Finally, \ref{reg_thm:ass4} follows easily as the function $m$ can be always assumed to be strictly convex (otherwise, one can add a strictly convex function to $m$). Therefore, Theorem \ref{res:iso_wpt_thm} applies so we conclude that for each $\theta \in (0,1]$ there is a unique solution $u^{\theta}$ as desired. Moreover, energy equality \eqref{global_energy_regular_problem} is valid.\\

\noindent Now, we intend to establish uniform estimates \ref{unif_est_1}--\ref{unif_est_4}. Let $m_1$ be a Young function such that $m_1(|\xi|) \leq M(t,x,\xi)$ as in point \ref{intro:defNf_control_isotropic} in Definition \ref{intro:def_Nfunc}. We estimate by using the H\"older inequality:
\begin{equation}
\begin{split}
\int_{\Omega_t} f(s,x) u^{\theta}(s,x) \diff s \diff x 
\leq 
\|f\|_{\infty} \int_{\Omega_t}  |u^{\theta}(s,x) | \diff s \diff x \leq
\|f\|_{\infty} \int_{\Omega_t}  \left(|u^{\theta}(s,x) |^2 + 1\right) \diff s \diff x.
\end{split}
\end{equation}
Using energy equality \eqref{global_energy_regular_problem} and noting that $A^{\theta}(s,x, \nabla u^{\theta}(s,x)) \cdot \nabla u^{\theta}(s,x) \geq 0$ we deduce that for a.e. $t \in (0,T)$
$$
\int_{\Omega} \left(u^{\theta}(t,x)\right)^2 \diff x \leq \int_{\Omega} \left(u_0(x)\right)^2 \diff x + \|f\|_{\infty} \int_0^t \int_{\Omega}  \left(|u^{\theta}(s,x) |^2 + 1\right) \diff x \diff s.
$$
Therefore, Gr\"onwall's lemma implies that $u^{\theta}$ is uniformly bounded in $L^{\infty}(0,T;L^2(\Omega))$. Moreover, \ref{intro:ass_on_A:coercgr} in Assumption \ref{intro:ass_on_A} leads to the estimate:
\begin{equation*}
\begin{split}
\int_{\Omega_t} M^*(s,x, A(s,x,\nabla u^{\theta}(s,x)))  \diff s \diff x +\int_{\Omega_t} M(s,x, \nabla u^{\theta}(s,x))  \diff s \diff x - \int_{\Omega_t} h(s,x)  \diff s \diff x \leq \\ \leq
c \int_{\Omega_t} A(s,x, \nabla u^{\theta}(s,x)) \cdot \nabla u^{\theta}(s,x) \diff s \diff x.
\end{split}
\end{equation*}
\noindent As $\int_{\Omega} \left(u^{\theta}(t,x)\right)^2 \diff x$ and $\int_{\Omega_t} f(s,x) u^{\theta}(s,x) \diff s \diff x$ are uniformly bounded, we deduce from energy equality \eqref{global_energy_regular_problem} that for a.e. $(t,x) \in \Omega_T$, the quantity
\begin{equation*}
\begin{split}
\int_{\Omega_t} M^*(s,x, A(s,x,\nabla u^{\theta}(s,x)))  \diff s \diff x + \, \int_{\Omega_t} M(s,x, \nabla u^{\theta}(s,x)) \, \diff s \diff x \,+\\ + 
\int_{\Omega_t} \theta \nabla_{\xi} m(|\nabla u^{\theta}(s,x)|) \cdot \nabla u^{\theta}(s,x)  \, \diff s \diff x \leq C(f,h,u_0),
\end{split}
\end{equation*}
the constant $C(f,h,u_0)$ is independent of $\theta$. Due to \ref{intro:propNfunct:item55} in Lemma \ref{intro:lem_prop_Nfunc}, we have that $\{\nabla u^{\theta}\}_{\theta \in (0,1]}$ is uniformly bounded in $\mus$ and $\{A(t,x,\nabla u^{\theta}) \}_{\theta}$ is uniformly bounded in $L_{M^*}(\Omega_T)$. Finally, using \eqref{useful_FYineq_eq} we deduce that sequence $\{ \theta m^*(\nabla_{\xi} m(|\nabla u^{\theta}|)) \}_{\theta \in (0,1]}$ is uniformly bounded in $L^1(\Omega_T)$.
\end{proof}
\noindent Thanks to the uniform bounds established in Lemma \ref{res:approximation_theta_lem}, we can now let $\theta \to 0$ in \eqref{res:regularized_problem}. The starting point for this limiting procedure is the observation that the approximative term vanishes in the limit, which is formulated in the next lemma.
\begin{lem}\label{res:limit_of_regularization}
Under notation and assumptions of Lemma \ref{res:approximation_theta_lem}, for any $\varphi: \Omega_T \mapsto \R^d$ such that $\varphi \in L^{\infty}(\Omega_T; \R^d)$, we have
$$
\lim_{\theta \to 0} \int_{\Omega_T} \theta \nabla_{\xi} m(|\nabla u^{\theta}|) \cdot \varphi \diff t \diff x = 0.
$$
\end{lem} 
\begin{proof}
This was also proved in \cite{chlebicka2019parabolic} but it was not formulated as a separate result so we provide the proof here. Consider $\Omega_T^R = \{(t,x) \in \Omega_T: |\nabla u^{\theta}| \leq R\}$ and write
\begin{equation}\label{res:splitting_integrals_aux_Estimate}
\int_{\Omega_T} \left|\theta \nabla_{\xi} m(|\nabla u^{\theta}|) \right| =
\int_{\Omega_T^R} \left|\theta \nabla_{\xi} m(|\nabla u^{\theta}|)  \right|+
\int_{\Omega_T\setminus \Omega_T^R} \left|\theta \nabla_{\xi} m(|\nabla u^{\theta}|) \right|.
\end{equation} 
For any $R > 0$, the first term converges to 0 as $\theta \to 0$. Note that by convexity, 
$$
m^*(\theta \nabla_{\xi} m(|\nabla u^{\theta}|)) \leq  m^*(\nabla_{\xi} m(|\nabla u^{\theta}|)) 
$$
so that due to \ref{intro:propNfunct:item5} in Lemma \ref{intro:lem_prop_Nfunc}, sequence $\{\theta \nabla_{\xi} m(|\nabla u^{\theta}|) \}_{\theta}$ is uniformly integrable. Therefore, as $R \to \infty$, the second term in \eqref{res:splitting_integrals_aux_Estimate} tends to 0 and the conclusion follows.
\end{proof}
\noindent The next result deals with the time derivatives of $u^{\theta}$ and will be used to deduce the pointwise convergence.
\begin{lem}\label{res:time_derivative_regularity}
Under notation and assumptions of Lemma \ref{res:approximation_theta_lem}, for every $\theta > 0$, we have $\partial_t u^{\theta} \in \left(W^1E_m(\Omega_T)\right)^*$ where $m$ is defined in Lemma \ref{res:approximation_theta_lem}. Moreover, for all $\varphi \in W^1E_m(\Omega_T)$ we have the following inequality:
\begin{equation}\label{res:time_der_bound}
\left(\partial_t u^{\theta}, \varphi \right) \leq C \| \varphi  \|_{W^1L_m},
\end{equation}
where the constant $C$ is independent of $\theta$.
\end{lem}
\begin{proof}
First, let $\varphi \in C^{\infty}_0((0,T) \times \Omega)$. By the weak formulation of \eqref{res:regularized_problem} we have
\begin{equation*}
\begin{split}
-\int_{\Omega_T} u^{\theta}(t,x) \partial_t \varphi(t,x) \diff t \diff x \,
+ &\int_{\Omega_T} A(t,x,\nabla u^{\theta}) \cdot \nabla \varphi(t,x) \diff t \diff x \,+\\ 
&+
\int_{\Omega_T} \theta_n \nabla_{\xi} m(|\nabla u^{\theta}|) \cdot \nabla \varphi \diff t \diff x= 
\int_{\Omega_T} f(t,x) \varphi(t,x) \diff t \diff x.
\end{split}
\end{equation*} 
Thus, we can estimate the left hand side using Lemma \ref{intro:Hold_Mus_inequality} as follows:
\begin{equation*}
\begin{split}
\left|\int_{\Omega_T} u^{\theta}(t,x) \partial_t \varphi(t,x) \diff t \diff x \right| \leq \, & \big\| A(t,x,\nabla u^{\theta}) \big\|_{L_{m^*}} \big\|\nabla \varphi \big\|_{L_m} \,+\\ 
&+
\theta_n \big\|  \nabla_{\xi} m(|\nabla u^{\theta}|) \big\|_{L_{m^*}} \big\|\nabla \varphi \big\|_{L_m} + 
\left|\Omega_T\right| m^*\left(\|f\|_{\infty}\right)  \big\| \varphi \big\|_{L_m}. 
\end{split}
\end{equation*} 
Note that $M(t,x,\xi) \leq m(|\xi|)$ implies $m^*(|\xi|) \leq M^*(t,x,\xi)$ and so,
$$
\big\| A(t,x,\nabla u^{\theta}) \big\|_{L_{m^*}} \leq \big\| A(t,x,\nabla u^{\theta}) \big\|_{L_{M^*}}.
$$
Therefore, we can use uniform bounds provided by Lemma \ref{res:approximation_theta_lem} and this (after application of the Poincar\'e inequality from Lemma \ref{intro:properties_orlicz_sob_lemma}) concludes the proof of \eqref{res:time_der_bound} for $\varphi \in C^{\infty}_0((0,T) \times \Omega)$. The general case follows by the density (in norm!) of $C^{\infty}_0((0,T) \times \Omega)$ in $W^1_0E_m(\Omega_T)$ (cf. \ref{prop_lemma:density} in Lemma \ref{intro:properties_orlicz_sob_lemma}).
\end{proof}
\noindent Finally, note that uniform bounds in Lemma \ref{res:approximation_theta_lem} guarantees the existence of subsequences (that we do not relabel) converging weakly-$\ast$ in appropriate spaces (cf. Lemma \ref{intro:weak_compactness}). We will also need stronger compactness provided by the following result.
\begin{lem}\label{res:rel_comp_conv_ae}
Under notation and assumptions of Lemma \ref{res:approximation_theta_lem}, the sequence $\{u^{\theta} \}_{\theta \in (0,1]}$ is relatively compact in $L^1(0,T; L^1(\Omega))$. In particular, it has a subsequence converging a.e. in $\Omega_T$. 
\end{lem}
\begin{proof}
We recall a version of Aubin-Lions Lemma (cf. \cite{simon1986compact}):\\

\noindent \underline{Aubin-Lions Lemma.} Let $X_0$, $X$ and $X_1$ be Banach spaces such that $X_0$ is compactly embedded in $X$ and $X$ is continuously embedded in $X_1$. Suppose that sequence of functions $\{f_n\}_{n \in \N}$ is bounded in $L^q(0,T; X)$ and $L^1(0,T; X_0)$. Moreover, assume that sequence of distributional time derivatives $\left\{ \partial_t f_n \right\}$ is bounded in $L^1(0,T;X_1)$. Then, $\{f_n\}_{n \in \N}$ is relatively compact in $L^p(0,T; X)$ for any $1 \leq p < q$.\\

\noindent We want to apply this result with $X_0 = W^{1,1}_0(\Omega)$, $X = L^1(\Omega)$ and $X_1 = W^{-2,r}(\Omega)$ for $r$ such that $W^{2,r}_0(\Omega)$ is continuously embedded in $C^1(\Omega)$ ($r > d$ is sufficient, cf. \cite[Corollary 7.11]{gilbarg2015elliptic}). 
\begin{itemize}
\item By Rellich-Kondrachov Theorem (or Arzela-Ascoli Theorem if $d=1$), $X_0$ is compactly embedded in $X$.
\item Let $f \in L^1(\Omega)$. Then, for $\varphi \in W^{2,r}_0(\Omega)$,
$$
\left|\int_{\Omega} f\, \varphi \right|  \leq \|f\|_{L^1} \| \varphi \|_{L^{\infty}} \leq C\|f\|_{L^1}  \| \varphi \|_{W^{2,r}},
$$
for some constant $C$ so that $X$ is continuously embedded in $X_1$.
\item Sequence $\{u^{\theta} \}_{\theta \in (0,1]}$ is uniformly bounded in $L^{\infty}(0,T;L^2(\Omega))$ and $\{\nabla u^{\theta} \}_{\theta \in (0,1]}$ is uniformly bounded in $L_{M^*}(\Omega_T)$. In particular, $\{u^{\theta} \}_{\theta \in (0,1]}$ is uniformly bounded in $L^1(0,T; W^{1,1}_0(\Omega))$ and $L^2(0,T; L^1(\Omega))$.
\item Let $\varphi \in L^{\infty}(0,T; W_0^{2,r}(\Omega))$ with $\big\| \varphi \big\|_{L^{\infty}(0,T; W_0^{2,r}(\Omega))} \leq 1$ and the plan is to prove that $\left(\partial_t u^{\theta}, \varphi \right)$ is uniformly bounded in $\varphi$ and $\theta \in (0,1]$. By the choice of $r$, there is a constant $C$ such that $\big|\varphi\big| \leq C$ and $\big|\nabla \varphi \big| \leq C$. In particular, $\varphi \in W^1_0E_m(\Omega_T)$ and $\big\| \varphi \big\|_{W^1L_m} \leq C$ for some possibly different constant $C$. Using Lemma \ref{res:time_derivative_regularity}, we establish assertion. By duality, this shows that $\partial_t u^{\theta}$ is uniformly bounded in $L^1(0,T; W^{-2,r}(\Omega))$.
\end{itemize}
Aubin-Lions Lemma implies that $\left\{ u^{\theta}\right\}_{\theta \in [0,1)}$ is relatively compact in $L^1(0,T; L^1(\Omega))$.
\end{proof}
\subsection{Equation $u_t = \DIV \alpha + f$ for $\alpha \in L_{M^*}(\Omega_T)$ and $f \in L^{\infty}(\Omega_T)$}\label{subsec:study_of_gen_par_eqn}
In this section we study the equation
$$
u_t = \DIV \alpha + f
$$
or more precisely, the following identity required to be satisfied for all $\varphi \in C^{\infty}_0([0,T) \times \Omega)$:
\begin{equation}\label{res:weak_form_for_sol_without_id}
\begin{split}
&-\int_{\Omega_T} u(t,x) \partial_t \varphi(t,x) \diff t \diff x 
- \int_{\Omega} u_0(x)\varphi(0,x) \diff x + \\ 
&\qquad \qquad \qquad \qquad \qquad+
\int_{\Omega_T} \alpha(t,x) \cdot \nabla \varphi(t,x) \diff t \diff x = 
\int_{\Omega_T} f(t,x) \varphi(t,x) \diff t \diff x,
\end{split}
\end{equation}
which is obtained in Section \ref{existproof:section_label} as the limit of \eqref{res:regularized_problem}. For $u: \Omega_T \to \R$ solving \eqref{res:weak_form_for_sol_without_id}, we write $\widetilde{u}$ to denote its extension:
\begin{equation}\label{res:extension_u}
\widetilde{u}(t,x) = 
\begin{cases}
0  & \mbox{ for } t > T,\\
u(t,x)  & \mbox{ for } t \in (0,T],\\
u_0(x) & \mbox{ for } t \leq 0.
\end{cases}
\end{equation}
We also extend $\alpha$ and $f$ to be zero for $t \in \R \setminus (0,T)$:
\begin{equation}\label{res:extension_Af}
\overline{\alpha}(t,x) = 
\begin{cases}
\alpha(t,x)  & \mbox{ for } t \in (0,T),\\
0 & \mbox{ for } t \in \R \setminus (0,T),
\end{cases}
\qquad \qquad 
\overline{f}(t,x) = 
\begin{cases}
f(t,x)  & \mbox{ for } t \in (0,T),\\
0 & \mbox{ for } t \in \R \setminus (0,T).
\end{cases}
\end{equation}
Our goal is to obtain some form of energy equality which will be crucial in developing the existence theory for \eqref{intro:parabolic_eq}. Classical approach (cf. \cite{chlebicka2019parabolic}) was based on appropriate mollification in space and time which required some continuity assumptions on $M(t,x,\xi)$ both in $t$ and $x$. Below, we show that mollification of the solution $u$ only in space has already Sobolev regularity in space and time.
\begin{lem}\label{res:Sobolev_regularity_solution_molinsp}
Suppose that $u \in V^M_T(\Omega)$, $\alpha \in L_{M^*}(\Omega_T)$ and $f \in L^{\infty}(\Omega_T)$. Consider extensions $\widetilde{u}$, $\overline{\alpha}$ and $\overline{f}$ defined in \eqref{res:extension_u} and \eqref{res:extension_Af}. Then,
\begin{equation}\label{res:extended_weak_form}
-\int_{\Omega} \int_{-T}^T \widetilde{u}(t,x) \partial_t \varphi(t,x) \diff t \diff x =
-\int_{\Omega} \int_{-T}^T  \overline{\alpha}(t,x) \cdot \nabla \varphi(t,x) \diff t \diff x = 
\int_{\Omega} \int_{-T}^T \overline{f}(t,x) \varphi(t,x) \diff t \diff x,
\end{equation}
for arbitrary $\varphi \in C^{\infty}_0((-T,T) \times \Omega)$. Moreover, $\widetilde{u}^{\varepsilon} \in W^{1,1}((-T,T)\times\Omega')$ where $\Omega' \Subset \Omega$. 
\end{lem}
\begin{proof}
To verify \eqref{res:extended_weak_form}, let $\varphi \in C^{\infty}_0((-T,T) \times \Omega)$. We compute using \eqref{res:weak_form_for_sol_without_id}:
\begin{equation*}
\begin{split}
& -\int_{\Omega} \int_{-T}^T \widetilde{u}(t,x) \partial_t \varphi(t,x) \diff t \diff x  = \\ 
& \qquad\qquad \qquad = -\int_{\Omega} \int_{-T}^0 \widetilde{u}(t,x) \partial_t \varphi(t,x) \diff t \diff x  -\int_{\Omega} \int_{0}^T \widetilde{u}(t,x) \partial_t \varphi(t,x) \diff t \diff x = \\  
& \qquad\qquad \qquad =
-\int_{\Omega}  u_0(x) \varphi(0,x) \diff x
-\int_{\Omega} \int_{0}^T u(t,x) \partial_t \varphi(t,x) \diff t \diff x =
\\
& \qquad\qquad \qquad =-\int_{\Omega} \int_{-T}^T  \overline{\alpha}(t,x) \cdot \nabla\varphi(t,x) \diff t \diff x + 
\int_{\Omega} \int_{-T}^T \overline{f}(t,x) \varphi(t,x) \diff t \diff x.
\end{split}
\end{equation*}
Mollifying \eqref{res:extended_weak_form} in space (by testing with mollified test function), we deduce $\partial_t u^{\varepsilon} \in L^1((-T,T) \times \Omega')$ proving the Sobolev regularity in time. Asserted regularity in space is obvious.
\end{proof}
\begin{rem}
Extension procedure above can be applied to obtain that $u^{\varepsilon} \in  
W^{1,1}((-M,T)\times\Omega')$ for any $0 < M < T$. However, we only need Sobolev regularity on $(-\delta,T)\times\Omega'$ for some $\delta > 0$ which can be arbitrarily small. 
\end{rem}
\begin{lem}[Local energy equality]\label{res:local_energy_equality_form}
Suppose that $u \in V^M_T(\Omega)$ is a solution to \eqref{res:weak_form_for_sol_without_id} with $\alpha \in L_{M^*}(\Omega_T)$, $f \in L^{\infty}(\Omega_T)$ and Assumption \ref{intro:ass_on_M} is satisfied. Then, for arbitrary $k \in \mathbb{N}$, for arbitrary $\psi \in C_0^{\infty}(\Omega)$ fulfilling $0\leq \psi \leq 1$ and for a.e. $t \in (0,T)$, the following energy equality is satisfied:
\begin{equation}\label{res:local_energy_equality}
\begin{split}
&\int_{\Omega} \psi(x) \big[G_k(u(t,x)) - G_k(u_0(x)) \big] \diff x=\\ 
& \qquad \qquad =
- \int_0^t \int_{\Omega} \alpha(s,x) \cdot \nabla\left[T_k(u(s,x))\, \psi(x) \right] \diff x \diff s + \int_0^t \int_{\Omega} f(s,x) \,T_k(u(s,x)) \,\psi(x) \diff x \diff s,
\end{split}
\end{equation}
where the function $G_k$ and the function $T_k$ are defined in Definition \ref{res:trunc_def}.
\end{lem}
\begin{proof}
For $s_1, s_2 \in \R$ and $\tau > 0$ we define the approximation of $\mathds{1}_{\left[s_1, s_2\right]}$:
$$
\gamma^{\tau}_{s_1, s_2} (s) = 
\begin{cases}
0 & \mbox{ for } s \leq s_1 - \tau \mbox{ or } s \geq s_2 + \tau,\\
1 & \mbox{ for } s \in [s_1, s_2],\\
\mbox{affine} & \mbox{ for } s \in [s_1 - \tau, s_1] \cup [s_2, s_2 + \tau].
\end{cases} 
$$
Let $\psi \in C_0^{\infty}(\Omega)$, $k \in \N$, $\varepsilon, \delta, \tau$ be small positive parameters and $\eta, \beta \in (0,T)$. Consider test function in \eqref{res:extended_weak_form}: 
$$
\varphi^{\delta,\tau,\varepsilon}_{\eta,\beta}(t,x) = \Big(S^{\delta}\Big(T_k(\widetilde{u}^{\varepsilon}(t,x)) \, \psi(x) \, \gamma^{\tau}_{-\eta, \beta} (t) \Big)\Big)^{\varepsilon} \in C^{\infty}_0((-T,T) \times \Omega),
$$
see Definitions \ref{res:mol_in_sp} and \ref{res:mol_in_ti} for mollification operators and Definition \ref{res:trunc_def} for truncation $T_k$. Note that since $\psi \in C_0^{\infty}(\Omega)$, mollification in space is well-defined for sufficiently small $\varepsilon>0$. \\

\noindent Now, we want to take limits in \eqref{res:extended_weak_form}: first $\delta \to 0$, then $\tau \to 0$ and finally $\varepsilon \to 0$. We denote:
\begin{equation*}
\begin{split}
&A^{\delta,\tau,\varepsilon}_{\eta,\beta} = -\int_{\Omega} \int_{-T}^T \widetilde{u}(t,x) \, \partial_t \varphi^{\delta,\tau,\varepsilon}_{\eta,\beta}(t,x) \diff t \diff x, \\
&B^{\delta,\tau,\varepsilon}_{\eta,\beta} = -\int_{\Omega} \int_{-T}^T  \overline{\alpha}(t,x) \cdot \nabla \varphi^{\delta,\tau,\varepsilon}_{\eta,\beta}(t,x) \diff t \diff x, \\
&C^{\delta,\tau,\varepsilon}_{\eta,\beta} = \int_{\Omega} \int_{-T}^T \overline{f}(t,x) \, \varphi^{\delta,\tau,\varepsilon}_{\eta,\beta} \diff t \diff x.
\end{split}
\end{equation*}
and we study each term separately.\\

\noindent \underline{Term $A^{\delta,\tau,\varepsilon}_{\eta,\beta}$.} Note that Sobolev derivatives and mollification commute so using Sobolev regularity in time from Lemma \ref{res:Sobolev_regularity_solution_molinsp}:
$$
A^{\delta,\tau,\varepsilon}_{\eta,\beta} = \int_{\Omega} \int_{-T}^T  \partial_t\widetilde{u}^{\varepsilon} (t,x) \Big(S^{\delta}\Big(T_k(\widetilde{u}^{\varepsilon}(t,x)) \psi(x) \gamma^{\tau}_{-\eta, \beta} (t) \Big)\Big)\diff t \diff x.
$$
Using Dominated Convergence (we still have $\varepsilon > 0$),
$$
\lim_{\tau \to 0} \lim_{\delta \to 0} A^{\delta,\tau,\varepsilon}_{\eta,\beta} =\int_{\Omega} \int_{-\eta}^{\beta}  \partial_t\widetilde{u}^{\varepsilon} (t,x) \, T_k(\widetilde{u}^{\varepsilon}(t,x)) \, \psi(x) \, \diff t \diff x =: A^{\varepsilon}_{\eta,\beta} .
$$
As function $G(s) = \int_0^s T_k(\sigma) \diff \sigma$ is $C^1$ with uniformly bounded derivative so standard chain rule for Sobolev maps \cite[Theorem 7.8]{gilbarg2015elliptic} together with Sobolev regularity in time from Lemma \ref{res:Sobolev_regularity_solution_molinsp} shows that $G(\widetilde{u}^{\varepsilon}(t,x)) \psi(x)$ is in $W^{1,1}((-T,T) \times \Omega)$, in particular it has Sobolev derivative in time. Moreover,
$$
\partial_t G(\widetilde{u}^{\varepsilon}(t,x)) = 
T_k(\widetilde{u}^{\varepsilon}(t,x))\, \partial_t\widetilde{u}^{\varepsilon}(t,x)
$$
Therefore, we can write:
$$
A^{\varepsilon}_{\eta,\beta} = \int_{\Omega} \int_{-\eta}^{\beta} \partial_t G\left(\widetilde{u}^{\varepsilon}(t,x) \right) \, \diff t  \, \psi(x) \, \diff x.
$$
Now, using absolute continuity on lines for Sobolev maps \cite[Theorem 4.21]{evans2015measure}, fundamental theorem of calculus applies for a.e. $x \in \Omega$ and $\eta, \beta \in (0,T)$ so we obtain
$$
A^{\varepsilon}_{\eta,\beta} = \int_{\Omega} \left[G_k\left(\widetilde{u}^{\varepsilon}(\beta,x) \right) -
G_k\left(\widetilde{u}^{\varepsilon}(-\eta,x) \right) \right]  \, \psi(x) \, \diff x.
$$
However, using definition of extension \eqref{res:extension_u}, this can be rewritten as
$$
A^{\varepsilon}_{\eta,\beta} = \int_{\Omega} \left[G_k\left(\widetilde{u}^{\varepsilon}(\beta,x) \right) -
G_k\left({u}^{\varepsilon}_0(x) \right) \right]  \, \psi(x) \, \diff x.
$$
Note that this step would not be achieved without extension for negative times as then, absolute continuity of Sobolev functions could be only applied for almost all times in $(0,T)$. Finally, using a.e. convergence of mollification and Dominated Convergence Theorem,
$$
\lim_{\varepsilon \to 0} A^{\varepsilon}_{\eta,\beta} = \int_{\Omega} \left[G_k\left(\widetilde{u}(\beta,x) \right) -
G_k\left({u}_0(x) \right) \right]  \, \psi(x) \, \diff x
$$
for almost all $\beta > 0$.\\

\noindent \underline{Term $B^{\delta,\tau,\varepsilon}_{\eta,\beta}$.} First, we use commutating properties of mollification to write:
\begin{equation*}
\begin{split}
B^{\delta,\tau,\varepsilon}_{\eta,\beta} &= -\int_{\Omega} \int_{-T}^T  \overline{\alpha}^{\varepsilon}(t,x) \cdot \nabla S^{\delta}\Big(T_k(\widetilde{u}^{\varepsilon}(t,x)) \psi(x) \gamma^{\tau}_{-\eta, \beta} (t) \Big) \diff t \diff x \\
&= \int_{\Omega} \int_{-T}^T  \DIV \overline{\alpha}^{\varepsilon}(t,x)\psi(x)~S^{\delta}\Big(T_k(\widetilde{u}^{\varepsilon}(t,x))  \gamma^{\tau}_{-\eta, \beta} (t) \Big) \diff t \diff x.
\end{split}
\end{equation*}
Note that as $\delta \to 0$ and $\tau \to 0$, $S^{\delta}\Big(T_k(\widetilde{u}^{\varepsilon}(t,x))  \gamma^{\tau}_{-\eta, \beta} (t) \Big) \to T_k(\widetilde{u}^{\varepsilon}(t,x))  \mathds{1}_{[-\eta, \beta]} (t) $ a.e. in $(-T,T) \times \Omega'$ for $\Omega' \Subset \Omega$. As $ \DIV \overline{\alpha}^{\varepsilon}(t,x) \psi(x) \in L^1(0,T; C_0^{\infty}(\Omega))$, we use Dominated Convergence Theorem to obtain
$$
\lim_{\tau \to 0} \lim_{\delta \to 0} B^{\delta,\tau,\varepsilon}_{\eta,\beta} = \int_{\Omega} \int_{-\eta}^\beta  \DIV \overline{\alpha}^{\varepsilon}(t,x)\psi(x)~T_k(\widetilde{u}^{\varepsilon}(t,x))  \diff t \diff x := B^{\varepsilon}_{\eta,\beta}.
$$  
Then, we write:
$$
B^{\varepsilon}_{\eta,\beta} = - \int_{\Omega} \int_{-\eta}^\beta  \overline{\alpha}(t,x)\cdot \nabla\left(\psi(x)T_k(\widetilde{u}^{\varepsilon}(t,x))\right)^{\varepsilon} \diff t \diff x.
$$
Due to Theorem \ref{res:approx_theorem}, $\nabla\left(\psi(x)T_k(\widetilde{u}^{\varepsilon}(t,x))\right)^{\varepsilon} \armus{M} \nabla\left(\psi(x)T_k(\widetilde{u}(t,x))\right)$ so using Corollary \ref{intro:hold_conv_inmod} we finally conclude
$$
B^{\eta,\tau,\varepsilon}_{\eta,\beta} \to - \int_{\Omega} \int_{-\eta}^\beta  \overline{\alpha}(t,x)\cdot \nabla\left(\psi(x)T_k(\widetilde{u}(t,x))\right) \diff t \diff x= 
- \int_{\Omega} \int_{0}^\beta  \alpha(t,x)\cdot \nabla\left(\psi(x)T_k(u(t,x))\right) \diff t \diff x.
$$

\noindent \underline{Term $C^{\delta,\tau,\varepsilon}_{\eta,\beta}$.} This is the easiest part. Note that $\varphi^{\delta,\tau,\varepsilon}_{\eta,\beta} \to T_k(\widetilde{u}(t,x)) \psi(x) \mathds{1}_{[-\eta, \beta]} (t) $ a.e. in $(-T,T)\times \Omega$ as $\delta \to 0$, $\tau \to 0$ and $\varepsilon \to 0$. Moreover, since $f \in L^{\infty}(\Omega_T)$ and $\left|\varphi^{\delta,\tau,\varepsilon}_{\eta,\beta} \right| \leq k$, we use Dominated Convergence Theorem to deduce
$$
C^{\delta,\tau,\varepsilon}_{\eta,\beta} \to 
 \int_{\Omega} \int_{-\eta}^{\beta} \overline{f}(t,x) T_k(\widetilde{u}(t,x)) \psi(x) \diff t \diff x =  \int_{\Omega} \int_{0}^{\beta} f(t,x) T_k(u(t,x)) \psi(x) \diff t \diff x.
$$
Finally, we obtain \eqref{res:local_energy_equality} for $t = \beta$ concluding the proof.
\end{proof}
\begin{rem}\label{res:local_energy_equality_form_reg_probl}
The same energy equality as \eqref{res:local_energy_equality} is satisfied by the solution to \eqref{res:regularized_problem}. Indeed, as the operator \eqref{res:regularized_problem_operator} is controlled by a Young function, Assumption \ref{intro:ass_on_M} is satisfied. Therefore, for $\psi \in C_0^{\infty}(\Omega)$ such that $0 \leq \psi(x) \leq 1$ and a.e. $t \in (0,T)$: 
\begin{equation}\label{res:local_energy_equality_regularized_problem}
\begin{split}
&\int_{\Omega} \psi(x) \left[G_k(u^{\theta}(t,x)) - G_k(u_0(x))\right] \diff x=\\ 
&  =
- \int_0^t \int_{\Omega} A_{\theta}(s,x,\nabla u^{\theta}) \cdot \nabla\left[T_k(u^{\theta}(s,x))\, \psi(x) \right] \diff x \diff s + \int_0^t \int_{\Omega} f(s,x) \,T_k(u^{\theta}(s,x)) \,\psi(x) \diff x \diff s.
\end{split}
\end{equation}
Note that $u^{\theta}$ also satisfies the global energy equality \eqref{global_energy_regular_problem}, see Lemma \ref{res:approximation_theta_lem}.
\end{rem}
\subsection{Local version of monotonicity method}\label{sect:monotonicity_trick}
The following procedure allows us to identify weak-$\ast$ limit of $A(t,x,\nabla u_n)$. We formulate here its local version and provide the proof that is almost identical to the global case presented in \cite[Lemma A.5]{chlebicka2019parabolic}. 
\begin{lem}\label{res:monot_trick}
Let $A$ satisfy Assumption \ref{intro:ass_on_A} and $M$ be an $N$-function. Assume that there are $\alpha \in L_{M^*}(\Omega_T; \R^d)$ and $\xi \in L_{M}(\Omega_T; \R^d)$ such that 
\begin{equation}\label{res:monot_trick_cond}
\int_{\Omega_T} \left(\alpha - A(t,x,\eta)\right)\cdot(\xi - \eta)\, \psi(x)\diff t \diff x \geq 0
\end{equation}
for all $\eta \in L^{\infty}(\Omega_T; \R^d)$ and $\psi \in C_0^{\infty}(\Omega)$ with $0 \leq \psi \leq 1$. Then,
$$
A(t,x,\xi) = \alpha(t,x) \mbox{ a.e. in } \Omega_T.
$$
\end{lem}
\begin{proof}
Consider subsets $\Omega_T^k = \{ (t,x) \in \Omega_T: |\xi(t,x)| \leq k \}$ and note that if $j < i$ then $\Omega_T^j \subset \Omega_T^i$. We use the assumption \eqref{res:monot_trick_cond} with $
\eta = \xi \mathds{1}_{\Omega_T^i} + h\,z\,\mathds{1}_{\Omega_T^j}$
where $h > 0$ and $z \in L^{\infty}(\Omega_T; \R^d)$ and we obtain
$$
\int_{\Omega_T} \left(\alpha - A(t,x,\xi \mathds{1}_{\Omega_T^i} + h\,z\,\mathds{1}_{\Omega_T^j})\right)\cdot(\xi - \xi \mathds{1}_{\Omega_T^i} - h\,z\,\mathds{1}_{\Omega_T^j}) \, \psi (x) \diff t \diff x \geq 0.
$$
Considering integral on $\Omega^i_T$ and $\Omega_T \setminus \Omega^i_T$ we deduce
$$
\int_{\Omega_T \setminus \Omega^i_T} \left(\alpha - A(t,x,0)\right)\cdot\xi \, \psi(x) \diff t \diff x  + 
h \int_{\Omega^j_T} \left(A(t,x,\xi  + h\,z) - \alpha\right)\cdot z \, \psi(x) \diff t \diff x \geq 0.
$$
Note that $A(s,x,0) = 0$ due to \ref{intro:assumA_vanish} in Assumption \ref{intro:ass_on_A}. Therefore, by integrability, the first term tends to $0$ as $i \to \infty$. Therefore,
$$
\int_{\Omega^j_T} \left(A(t,x,\xi  + h\,z) - \alpha\right)\cdot z \, \psi(x) \diff t \diff x \geq 0.
$$
Now, we want to let $h \to 0$. We have convergence $A(t,x,\xi  + h\,z) \to  A(t,x,\xi)$ due to \ref{intro:ass_on_A:continuity} in Assumption \ref{intro:ass_on_A}. Moreover, $\xi  + h\,z$ is uniformly bounded on $\Omega^j_T$. Therefore, \ref{intro:propNfunct:item6} in Lemma \ref{intro:lem_prop_Nfunc} implies:
$$
\int_{\Omega^j_T} \left(A(t,x,\xi) - \alpha\right)\cdot z \, \psi(x) \diff t \diff x \geq 0.
$$
Finally, choosing $z(t,x) = -\frac{A(t,x,\xi) - \alpha(t,x)}{|A(t,x,\xi) - \alpha(t,x)|}\mathds{1}_{A(t,x,\xi) - \alpha(t,x) \neq 0}$, we deduce
$$
A(t,x,\xi) = \alpha(t,x) \qquad \mbox{ for a.e.} (t,x) \in \Omega^j_T \cap \supp \psi.
$$
Since $j$ and $\psi$ are arbitrary, the assertion follows.
\end{proof}
\section{Proof of existence result (Theorem \ref{mainresult:existence})}\label{existproof:section_label}
\noindent Consider sequence of solutions $\{u^{\theta} \}_{\theta \in (0,1]}$ to the regularized problem \eqref{res:regularized_problem}. Using Lemma \ref{res:rel_comp_conv_ae} as well as uniform bounds from Lemmata \ref{res:approximation_theta_lem} and \ref{intro:weak_compactness}, we can extract a subsequence denoted with $u_n := u^{\theta_n}$ and $\theta_n \to 0$ such that:
\begin{itemize}
\item $u_n \to u$ in $L^1(0,T; L^1(\Omega))$ and a.e. in $\Omega_T$,
\item $u_n \wstar u$ weakly-$\ast$ in $L^{\infty}(0,T;L^2(\Omega))$,
\item $\nabla u_n \wstar \nabla u$ weakly-$\ast$ in $L_M(\Omega_T)$,
\item $u_n \rightharpoonup u$ weakly in $L^1(0,T; W^{1,1}(\Omega))$,
\item $A(\cdot,\cdot, \nabla u_n)\wstar \alpha$ weakly-$\ast$ in $L_{M^*}(\Omega_T)$,
\end{itemize}
for some $u \in V^M_T(\Omega)$ and $\alpha \in L_{M^*}(\Omega_T)$.\\

\noindent For solutions to the regularized problem \eqref{res:regularized_problem} we have the weak formulation. Namely, for all $\varphi \in C^{\infty}_0([0,T) \times \Omega)$:
\begin{equation}\label{exist_proof:weak_form_bef_pass}
\begin{split}
-\int_{\Omega_T} u_n(t,x) \partial_t \varphi(t,x) \diff t \diff x \, 
-  &\int_{\Omega} u_0(x)\varphi(0,x) \diff x + \int_{\Omega_T} A(t,x,\nabla u_n) \cdot \nabla \varphi(t,x) \diff t \diff x \,+\\ 
&+
\int_{\Omega_T} \theta_n \nabla_{\xi} m(|\nabla u_{n}|) \cdot \nabla \varphi \diff t \diff x= 
\int_{\Omega_T} f(t,x) \varphi(t,x) \diff t \diff x.
\end{split}
\end{equation} 
Using Lemma \ref{res:limit_of_regularization}, we can pass to the limit with $n\to \infty$ (or $\theta_n \to 0$) in \eqref{exist_proof:weak_form_bef_pass} to obtain:
\begin{equation}\label{exist_proof:weak_form_after_pass}
\begin{split}
-\int_{\Omega_T} u(t,x) \partial_t \varphi(t,x) \diff t \diff x \,
- &\int_{\Omega} u_0(x)\varphi(0,x) \diff x =\\ 
&= -\, \int_{\Omega_T} \alpha \cdot \nabla \varphi(t,x) \diff t \diff x \,+
\int_{\Omega_T} f(t,x) \varphi(t,x) \diff t \diff x.
\end{split}
\end{equation} 
Thanks to \eqref{exist_proof:weak_form_after_pass}, the theory from Section \ref{subsec:study_of_gen_par_eqn} can be applied and by using Lemma \ref{res:local_energy_equality_form} we obtain that for $\psi \in C_0^{\infty}(\Omega)$ with $0\leq \psi \leq 1$ and a.e. $t \in (0,T)$:
\begin{equation}\label{res:local_energy_equality_in_the_existence_proof}
\begin{split}
\int_{\Omega} \psi(x) \left[G_k(u(t,x)) - G_k(u_0(x)) \right] \diff x= - \int_0^t \int_{\Omega} \alpha(s,x) \cdot \nabla \big(T_k(u(s,x))\big)\, \psi(x) \diff x \diff s\\ 
- \int_0^t \int_{\Omega} \alpha(s,x) \cdot \nabla\psi(x)\, T_k(u(s,x)) \diff x \diff s + \int_0^t \int_{\Omega} f(s,x) \,T_k(u(s,x)) \,\psi(x) \diff x \diff s.
\end{split}
\end{equation}
Due to Remark \ref{res:local_energy_equality_form_reg_probl}, a similar energy equality holds for sequence $\{ u_n \}_{n\in \N}$:
\begin{equation}\label{res:local_energy_equality_in_the_existence_proof_approx_problem}
\begin{split}
&\int_{\Omega} \psi(x) \left[G_k(u_n(t,x)) - G_k(u_0(x)) \right] \diff x=\\ 
&  =
- \int_0^t \int_{\Omega} A_{\theta_n}(s,x,\nabla u_n) \cdot \nabla\left[T_k(u_n)\, \psi(x) \right] \diff x \diff s + \int_0^t \int_{\Omega} f(s,x) \,T_k(u_n) \,\psi(x) \diff x \diff s.
\end{split}
\end{equation}
We note that the term with operator $A_{\theta_n}(s,x,\nabla u_n)$ can be decomposed into four parts:
\begin{itemize}
\item $\int_0^t \int_{\Omega} A(s,x,\nabla u_n) \cdot \nabla\big(T_k(u_n)\big)\, \psi(x) \diff x \diff s $ 
\item $\int_0^t \int_{\Omega} A(s,x,\nabla u_n) \cdot \nabla\psi(x) \, T_k(u_n) \,  \diff x \diff s $ which, due to $A(s,x,\nabla u_n) \wstar \alpha$, $u_n \to u$ a.e. and Dominated Convergence Theorem, converges to $\int_0^t \int_{\Omega} \alpha \cdot \nabla\psi(x) \, T_k(u) \,  \diff x \diff s $,
\item $\int_0^t \int_{\Omega} \theta_n \nabla_{\xi} m(|\nabla u_n|) \cdot \nabla\big(T_k(u_n)\big)\, \psi(x) \diff x \diff s $, which is nonnegative, 
\item $\int_0^t \int_{\Omega} \theta_n \nabla_{\xi} m(|\nabla u_n|) \cdot \nabla\psi(x) \, T_k(u_n) \,  \diff x \diff s $, converging to $0$ due to Lemma \ref{res:limit_of_regularization}.
\end{itemize}
Therefore, \eqref{res:local_energy_equality_in_the_existence_proof_approx_problem} implies:
\begin{equation*}
\begin{split}
& \limsup_{n \to \infty} \int_0^t \int_{\Omega} A(s,x,\nabla u_n) \cdot \nabla\big(T_k(u_n(s,x))\big)\, \psi(x) \diff x \diff s  \leq
 \\
&  \quad \leq  - \,\int_0^t \int_{\Omega} \alpha \cdot \nabla\psi(x) \, T_k(u(s,x)) \,  \diff x \diff s - \int_{\Omega} \psi(x) \left[G_k(u(t,x)) - G_k(u_0(x)) \right] \diff x \\
& \quad \quad + \int_0^t \int_{\Omega} f(s,x) \,T_k(u(s,x)) \,\psi(x) \diff x \diff s,
\end{split}
\end{equation*}
which combined with \eqref{res:local_energy_equality_in_the_existence_proof} yields:
\begin{equation}\label{proof:almost_monotonicity}
\limsup_{n \to \infty} \int_0^t \int_{\Omega} A(s,x,\nabla u_n) \cdot \nabla\big(T_k(u_n)\big)\, \psi(x) \diff x \diff s  \leq \int_0^t \int_{\Omega} \alpha(s,x) \cdot \nabla \big(T_k(u)\big)\, \psi(x) \diff x \diff s.
\end{equation}
%Finally note that $A(s,x,\xi) = 0$ if and only if $\xi = 0$ due to \eqref{intro:ass_on_A:coercgr} in Assumption \ref{intro:ass_on_A}. 
Now, let $\alpha_k = \alpha \mathds{1}_{|u(t,x)| < k}$. We claim that for any $k \in \mathbb{N}$, $\eta \in L^{\infty}\left(\Omega_T; \R^d\right)$, $\psi \in C_0^{\infty}(\Omega)$ such that $0 \leq \psi \leq 1$ and a.e. $t \in (0,T)$:
\begin{equation}\label{proof:claim_monotonicity}
\int_{\Omega} \int_0^t \left(\alpha_k  - A(s,x,\eta)\right)\cdot(\nabla T_k(u) - \eta)\, \psi(x)  \diff s \diff x \geq 0.
\end{equation}
Indeed, by monotonicity (\ref{intro:assumA_mono} in Assumption \ref{intro:ass_on_A}) we have that
$$
\int_{\Omega} \int_0^t \left(A(s,x,\nabla T_k(u_n))  - A(s,x,\eta)\right)\cdot(\nabla T_k(u_n) - \eta) \, \psi(x) \diff s \diff x \geq 0.
$$
By denoting $\Omega_t = (0,t) \times \Omega$, we see that:
\begin{itemize}
\item $\int_{\Omega_t}A(s,x,\nabla T_k(u_n))\cdot \nabla T_k(u_n) \, \psi\diff s \diff x = \int_{\Omega_T}A(s,x,\nabla u_n)\cdot \nabla T_k(u_n)\, \psi \diff s \diff x $ since we have $\nabla\left[T_k(u_n)\right] = \nabla u_n \, \mathds{1}_{|u_n| < k}$,
\item $\int_{\Omega_t}A(s,x,\nabla T_k(u_n))\cdot \eta \, \psi \diff s \diff x\to \int_{\Omega_t} \alpha \mathds{1}_{|u(s,x)| < k}\cdot \eta \, \psi \diff s \diff x= \int_{\Omega_t} \alpha_k \cdot \eta \, \psi \diff s \diff x$. Indeed, we can write $A(s,x,\nabla T_k(u_n)) = A(s,x,\nabla u_n)\mathds{1}_{|u_n(s,x)| < k}$ and pass to the limit with $n$ using $A(s,x,\nabla u_n) \wstar \alpha(s,x)$ and $u_n \to u$ a.e.,
\item $\int_{\Omega_t} A(s,x,\eta) \cdot \nabla T_k(u_n) \, \psi \diff s \diff x \to \int_{\Omega_t} A(s,x,\eta) \cdot \nabla T_k(u) \, \psi \diff s \diff x$ due to $\nabla u_n \wstar \nabla u$ and $u_n \to u$ a.e.
\end{itemize}
Therefore, \eqref{proof:claim_monotonicity} follows. By monotonicity trick (Lemma \ref{res:monot_trick}), $\alpha_k(t,x) = A(t,x,\nabla T_k(u))$ for any $k \in \N$ and this finally implies $\alpha = A(t,x,u)$ concluding the proof of existence.\\

\noindent Finally, to establish global energy inequality \eqref{theorem:energyinequality}, we note that for a.e. $t \in (0,T)$
\begin{equation}\label{wlsc_l2}
\int_{\Omega} u^2(t,x) \diff x \leq \liminf_{n \to \infty} \int_{\Omega} (u^n(t,x))^2 \diff x 
\end{equation}
as $L^2$ norm is weakly lower semicontinuous (since it is strongly continuous and convex). We claim that 
\begin{equation}\label{wlsc_Aform}
\liminf_{n \to \infty} \int_0^t \int_{\Omega} A(s,x, \nabla u_n(s,x)) \cdot \nabla u_n(s,x) \, \diff x \diff s \geq \int_0^t \int_{\Omega} A(s,x, \nabla u(s,x)) \cdot \nabla u(s,x) \, \diff x \diff s.
\end{equation}
Indeed, let $k \in \mathbb{N}$. We can write
\begin{equation*}
\begin{split}
&\int_0^t \int_{\Omega} A(s,x, \nabla u_n(s,x)) \cdot \nabla u_n(s,x) \, \diff x \diff s = 
\\
& \quad = \int_0^t \int_{\Omega} \Big[A(s,x, \nabla u_n(s,x)) - A(s,x, \nabla u(s,x) \mathds{1}_{|\nabla u| \leq k})\Big] \cdot \Big[\nabla u_n(s,x) - \nabla u(s,x) \mathds{1}_{|\nabla u| \leq k} \Big] \, \diff x \diff s \\
& \quad + \int_0^t \int_{\Omega} A(s,x, \nabla u(s,x) \mathds{1}_{|\nabla u| \leq k}) \cdot \Big[\nabla u_n(s,x) - \nabla u(s,x) \mathds{1}_{|\nabla u| \leq k} \Big] \, \diff x \diff s \\
& \quad + \int_0^t \int_{\Omega} A(s,x, \nabla u_n(s,x)) \cdot \nabla u(s,x) \mathds{1}_{|\nabla u| \leq k},
\end{split}
\end{equation*}
where the first term is nonnegative due to \ref{intro:assumA_mono} in Assumption \ref{intro:ass_on_A}. Recall that we already know that $\nabla u_n \wstar \nabla u$ weakly-$\ast$ in $L_M(\Omega_T)$ and $A(\cdot,\cdot, \nabla u_n)\wstar A(\cdot,\cdot, \nabla u)$ weakly-$\ast$ in $L_{M^*}(\Omega_T)$. Lemma \ref{intro:lem_bound_on_A} implies that the map $(s,x)\mapsto A(s,x, \nabla u(s,x) \mathds{1}_{|\nabla u| \leq k})$ is bounded. Therefore,
\begin{equation*}
\begin{split}
& \liminf_{n\to \infty}\int_0^t \int_{\Omega} A(s,x, \nabla u_n(s,x)) \cdot \nabla u_n(s,x) \, \diff x \diff s \geq \\
& \qquad \qquad \qquad\geq \int_0^t \int_{\Omega} A(s,x, \nabla u(s,x) \mathds{1}_{|\nabla u| \leq k}) \cdot \nabla u(s,x) \mathds{1}_{|\nabla u| \geq k} \, \diff x \diff s \\
& \qquad \qquad \qquad + \int_0^t \int_{\Omega} A(s,x, \nabla u(s,x)) \cdot \nabla u(s,x) \mathds{1}_{|\nabla u| \leq k},
\end{split}
\end{equation*}
where the first term vanished due to presence of two characteristic functions $\mathds{1}_{|\nabla u| \geq k}$ and $\mathds{1}_{|\nabla u| \leq k}$. Finally, we let $k \to \infty$ and deduce \eqref{wlsc_Aform}.\\

\noindent By energy equality for the regularized problem \eqref{global_energy_regular_problem}, we have:
\begin{equation*}
\begin{split}
&\int_{\Omega} \big[(u_n(t,x))^2 - (u_0(x))^2 \big] \diff x =
-\, \int_0^t \int_{\Omega} A(s,x, \nabla u_n(s,x)) \cdot \nabla u_n(s,x) \, \diff x \diff s
\\ &  \qquad \qquad \qquad 
-\, \int_0^t \int_{\Omega} \theta_n \nabla_{\xi} m(|\nabla u_n|) \cdot \nabla u_n(s,x) \, \diff x \diff s + \int_0^t \int_{\Omega} f(s,x) \,u_n(s,x) \, \diff x \diff s.
\end{split}
\end{equation*}
We note that $\int_0^t \int_{\Omega} \theta_n \nabla_{\xi} m(|\nabla u_n|) \cdot \nabla u_n(s,x) \, \diff x \diff s \geq 0$ so that
\begin{equation*}
\begin{split}
\int_{\Omega} \big[(u_n(t,x))^2 - (u_0(x))^2 \big] \diff x \leq &
-\, \int_0^t \int_{\Omega} A(s,x, \nabla u_n(s,x)) \cdot \nabla u_n(s,x) \, \diff x \diff s
\\ & 
 + \int_0^t \int_{\Omega} f(s,x) \,u_n(s,x) \, \diff x \diff s.
\end{split}
\end{equation*}
Using \eqref{wlsc_l2} and \eqref{wlsc_Aform}, we let $n \to \infty$ and conclude the proof of the energy inequality \eqref{theorem:energyinequality} for a.e. $t \in (0,T)$. Finally, as the map $[0,T) \ni t \mapsto u(t,\cdot) \in L^2(\Omega)$ is weakly continuous, energy inequality holds for all $t \in [0,T)$.  
\section{Proof of uniqueness result (Theorem \ref{mainresult:uniqueness})}
\noindent To obtain the uniqueness of a weak solution, it is standard in the theory of parabolic equations to test the equation for the difference of solutions with the difference of solutions itself. In the Musielak-Orlicz framework, it is unfortunately not so straightforward. In fact, we want to improve the result of Lemma \ref{res:local_energy_equality_form}, where we showed the local energy equality, to the global energy equality, i.e. we want to remove the presence of the cut-off function. Next lemma shows that under the additional structural hypothesis on $M$ (the radial symmetry), such procedure can be made rigorous.
\begin{lem}\label{proof:uniqueness:crucial_lemma}
Let $\Omega$ be a Lipschitz domain. Suppose that the $N$-function $M$ is isotropic (as in assumptions of Theorem \ref{mainresult:uniqueness}) and Assumption \ref{intro:ass_on_M} is satisfied. Then, there is a family of functions $\left\{\psi_j \right\}_{j \in \N}$ compactly supported in $\Omega$ and fulfilling $\psi_j \to 1$ as $j \to \infty$, such that if $u \in L^{\infty}(0,T; L^{\infty}(\Omega)) \cap L^1(0,T; W^{1,1}_0(\Omega))$ with $\nabla u \in \mus$, we have
$$
\int_0^t \int_{\Omega} M\left(s,x,\left|\frac{\nabla\psi_{j}(x) \, u(s,x)}{C_{u}}\right|\right) \,\diff x \diff s \to 0 \mbox{ as } j \to \infty,
$$
where the constant $C_{u}$ can be chosen as $C_{u} = C \| \nabla u \|_{L_M}$ where $C$ depends only on $\Omega$.
\end{lem}
\begin{proof}
Since $\Omega$ is a Lipschitz domain, we can flatten the boundary locally with bi-Lipschitz homeomorphisms so that using appropriate partition of unity, $\partial \Omega$ can be assumed to be flat. This argument relies heavily on the isotropy of $M$ as otherwise it is not clear if $\nabla u \in L_M(\Omega_T)$ implies $\nabla \left(u \circ \Psi\right) \in L_M(\Omega_T)$ for some bi-Lipschitz homeomorphism $\Psi$. \\ 

\noindent Let $\Omega_j = \left\{ x \in \Omega: \mbox{dist}(x, \partial \Omega) > \frac{1}{j} \right\}$ so that $\Omega_j \nearrow \Omega$ as $j \to \infty$. Moreover, let $\psi_j \in C_0^{\infty}(\Omega)$ such that $\psi_j = 1$ on $\Omega_j$. Note that $\nabla \psi_j = 0$ on $\Omega_j$ and $|\nabla \psi_j| \leq C j$ for some constant $C$. We cover $\Omega \setminus \Omega_j$ with the family of disjoint cubes $\{ Q_m^{j} \}_{m=1}^{N_{j}}$ with edge of length $\frac{1}{j}$. Then, we write for some constant $C_u$ to be chosen later:
\begin{equation}\label{uniqueness_bound}
\begin{split}
& \int_0^t \int_{\Omega} M\left(s,x,\left|\frac{\nabla\psi_{j}(x) \, u(s,x)}{C_{u}}\right|\right) \,\diff x \diff s = \\ 
&  \qquad \qquad =\int_0^t \int_{\Omega \setminus \Omega_j} M\left(s,x,\left|\frac{\nabla\psi_{j}(x) \, u(s,x)}{C_{u}}\right|\right) \,\diff x \diff s \\
&  \qquad \qquad \leq \sum_{m=1}^{N_j} \int_0^t \int_{Q_m^j \cap \Omega \setminus \Omega_j} M\left(s,x,\left|\frac{\nabla\psi_{j}(x) \, u(s,x)}{C_{u}}\right|\right) \,\diff x \diff s \\
  & \qquad \qquad \leq
\sum_{i=1}^{N_j} \int_0^t \int_{Q_m^j \cap \Omega \setminus \Omega_j} \frac{M\left(s,x,\left|\frac{\nabla\psi_{j}(x) \, u(s,x)}{C_{u}}\right|\right)}{M^{**}_{Q_m^j }\left(s,\left|\frac{\nabla\psi_{j}(x) \, u(s,x)}{C_{u}}\right| \right)} M^{**}_{Q_m^j }\left(s,\left|\frac{\nabla\psi_{j}(x) \, u(s,x)}{C_{u}}\right| \right) \,\diff x \diff s.
\end{split}
\end{equation}
Note that $\left|\frac{\nabla\psi_{j}(x) \, u(s,x)}{C_{u}}\right| \leq \frac{j \|u\|_{\infty} }{C_u}$ so that we can apply Assumption \ref{intro:ass_on_M} and deduce:
$$
\limsup_{j \to \infty} \frac{M\left(s,x,\left|\frac{\nabla\psi_{j}(x) \, u(s,x)}{C_{u}}\right|\right)}{M^{**}_{Q_m^j }\left(s,\left|\frac{\nabla\psi_{j}(x) \, u(s,x)}{C_{u}}\right| \right)} \leq C.
$$
Therefore, \eqref{uniqueness_bound} reads: 
\begin{equation}\label{uniqueness_bound_after_holcont}
\int_0^t \int_{\Omega} M\left(s,x,\left|\frac{\nabla\psi_{j}(x) \, u(s,x)}{C_{u}}\right|\right) \,\diff x \diff s \leq C  \sum_{i=1}^{N_j} \int_0^t \int_{Q_m^j \cap \Omega \setminus \Omega_j} M^{**}_{Q_m^j }\left(s,\left|\frac{\nabla\psi_{j}(x) \, u(s,x)}{C_{u}}\right| \right) \,\diff x \diff s.
\end{equation}
Now, for any $x = (x_1, x_2, ..., x_d) \in Q_m^j$, we write $x^* = (x_1, x_2, ..., 0)$ for its projection on the face of the cube $Q_m^j$ sticking to the boundary $\partial \Omega$ (see Figure \ref{pic_boundaryOmega}). Note that we assumed that the face of the cube is perpendicular to the axis of the last variable $x_d$ which is not restrictive and can be obtained by choosing appropriate straightening bi-Lipschitz homeomorphisms.\\

\begin{figure}
\begin{tikzpicture}

%coordinate system and dashed lines
\draw[line width=0.4mm,->] (0,0) -- (7.7,0) node[anchor=north west] {$\R^{d-1}$ (variables $x_1$, ..., $x_{d-1}$)};
\draw[line width=0.4mm,->] (0,0) -- (0,4.4) node[anchor=south east] {$x_d$};

%cubes
\filldraw[fill=white, draw=black] (1.0,0) rectangle (2.0,1);
\filldraw[fill=lightgray, draw=black] (2.0,0) rectangle (3.0,1);
\filldraw[fill=lightgray, draw=black] (3.0,0) rectangle (4.0,1);
\filldraw[fill=lightgray, draw=black] (4.0,0) rectangle (5.0,1);
\filldraw[fill=lightgray, draw=black] (5.0,0) rectangle (6.0,1);
\filldraw[fill=white, draw=black] (6.0,0) rectangle (7.0,1);
%arrow 1/j
\draw[line width=0.4mm,<->] (7.3,0) -- (7.3,1);
\node at (7.6, 0.5) {$\frac{1}{j}$};

%boundary in  red
\draw [line width=0.4mm, draw=red] plot [smooth, tension=1] coordinates {  (1.3,0) (2.6,3.6) (5,3.1) (6.8,0)};
\draw[line width=0.6mm, draw=red] (1.3,0) -- (6.8,0);
\node at (5.3, 3.4) {{$\partial{\Omega}$}};

\end{tikzpicture}

\caption{The boundary $\partial \Omega$ with some part of it flattened after change of coordinates. Gray cubes from the family $\left\{Q^j_m\right\}_{m=1}^{N_j}$ correspond to the area that is relevant for further computations after application of partition of unity.}
\centering
\label{pic_boundaryOmega}
\end{figure}
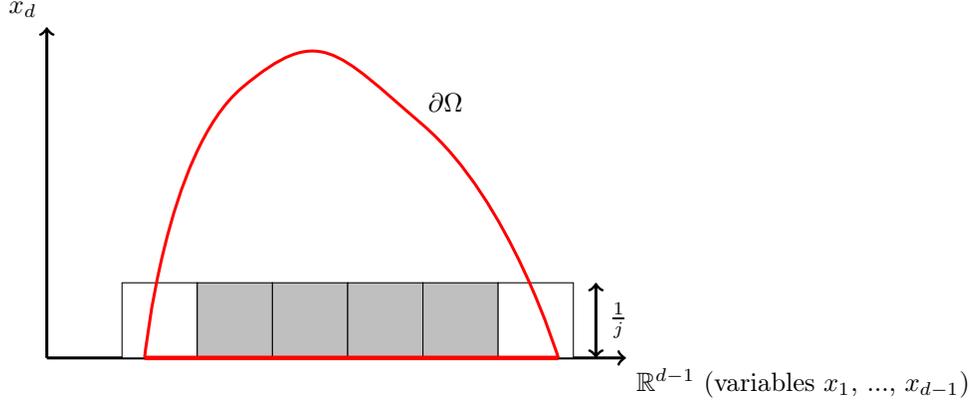

\noindent For a.e. $x \in Q_m^j$, using absolute continuity on lines for Sobolev maps (cf. Theorem 4.21 in \cite{evans2015measure}), we can write:
$$
u(s,x) = \int_0^{x_d} u_{x_d}(s, x_1, x_2, ..., r) \, \diff r
$$
where $u_{x_d}$ denotes derivative with respect to the last variable. Note that since $|x_d| \leq \frac{1}{j}$, $|u(s,x)|$ can be bounded as
$$
|u(s,x)| \leq  \int_0^{x_d} |u_{x_d}(s, x_1, x_2, ..., x_{d-1}, r)| \, \diff r
\leq \int_0^{\frac{1}{j}} |\nabla u(s, x_1, x_2, ..., x_{d-1}, r)| \, \diff r.
$$
Using this inequality in \eqref{uniqueness_bound_after_holcont}, we can continue as follows:
\begin{equation*}
\begin{split}
&\sum_{i=1}^{N_j} \int_0^t \int_{Q_m^j \cap \Omega \setminus \Omega_j} M^{**}_{Q_m^j }\left(s,\left|\frac{\nabla\psi_{j}(x) \, u(s,x)}{C_{u}}\right| \right) \,\diff x \diff s \leq \\ 
& \qquad \qquad \qquad \leq \sum_{i=1}^{N_j} \int_0^t \int_{Q_m^j \cap \Omega \setminus \Omega_j} M^{**}_{Q_m^j }\left(s, \frac{\left|\nabla\psi_{j}(x)\right|}{C_{u}} \int_0^{\frac{1}{j}} |\nabla u(s, x_1, x_2, ..., x_{d-1}, r)| \, \diff r \right) \,\diff x \diff s \\
& \qquad \qquad \qquad =\sum_{i=1}^{N_j} \int_0^t \int_{Q_m^j \cap \Omega \setminus \Omega_j} M^{**}_{Q_m^j }\left(s, \frac{C}{C_{u}} \, j  \int_0^{\frac{1}{j}} |\nabla u(s, x_1, x_2, ...,x_{d-1}, r)| \, \diff r \right) \,\diff x \diff s\\
& \qquad \qquad \qquad \leq \sum_{i=1}^{N_j}\int_0^{\frac{1}{j}} \int_0^t \int_{Q_m^j \cap \Omega \setminus \Omega_j} j \, M^{**}_{Q_m^j }\left(s,\frac{C}{C_u}  \left|\nabla u(s, x_1, x_2, ..., x_{d-1}, r) \right|  \right) \,\diff x \diff s \diff r
\end{split}
\end{equation*}
where we used the bound $\left|\nabla\psi_{j}(x)\right| \leq Cj$ and Jensen's inequality. Note that the integrand does not depend on $x_d$ and so, the integral over this variable cancels with the factor $j$. Finally, as cube has edge of length $\frac{1}{j}$, Fubini's theorem implies
\begin{equation*}
\begin{split}
&\sum_{i=1}^{N_j} \int_0^t \int_{Q_m^j \cap \Omega \setminus \Omega_j} M^{**}_{Q_m^j }\left(s,\left|\frac{\nabla\psi_{j}(x) \, u(s,x)}{C_{u}}\right| \right) \,\diff x \diff s \leq \\ 
& \qquad \qquad \qquad = \sum_{i=1}^{N_j} \int_0^t \int_{Q_m^j \cap \Omega \setminus \Omega_j} \, M^{**}_{Q_m^j }\left(s,\frac{C}{C_u}  \left|\nabla u(s, x_1, x_2, ..., x_{d-1}, x_d) \right|  \right) \,\diff x \diff s\\
& \qquad \qquad \qquad \leq \sum_{i=1}^{N_j} \int_0^t \int_{Q_m^j \cap \Omega \setminus \Omega_j} \, M\left(s,x,\frac{C}{C_u}  \left|\nabla u(s, x_1, x_2, ..., x_{d-1}, x_d) \right|  \right) \,\diff x \diff s\\
& \qquad \qquad \qquad = \int_0^t \int_{\Omega \setminus \Omega_j} \, M\left(s,x,\frac{C}{C_u}  \left|\nabla u(s, x) \right|  \right) \,\diff x \diff s.
\end{split}
\end{equation*}
Now, as $\nabla u \in L_M(\Omega_T)$, we can choose $C_u = C \| \nabla u \|_{L_M}$ so that the integral
$$
\int_0^t \int_{\Omega} M\left(s,x,\frac{C}{C_u}  \left|\nabla u(s, x) \right|  \right) \,\diff x \diff s
$$
is finite and the conclusion follows by $\Omega_j \nearrow \Omega$ as $j \to \infty$.
\end{proof}
\begin{lem}[Global energy equality]
Under assumptions of Lemma \ref{res:local_energy_equality_form} and Theorem \ref{mainresult:uniqueness}, the following energy equality is satisfied for a.e. $t \in (0,T)$:
\begin{equation}\label{res:global_energy_equality}
\begin{split}
&\int_{\Omega} \left[G_k(u(t,x)) - G_k(u_0(x)) \right] \diff x=\\ 
& \qquad \qquad =
- \int_0^t \int_{\Omega} \alpha(s,x) \cdot \nabla\left[T_k(u(s,x)) \right] \diff x \diff s + \int_0^t \int_{\Omega} f(s,x) \,T_k(u(s,x)) \, \diff x \diff s.
\end{split}
\end{equation}
\end{lem}
\begin{proof}
The main idea is to consider local energy equality \eqref{res:local_energy_equality} with a sequence of cut-off functions $\left\{\psi_j\right\}_{j \in \N}$ from Lemma \ref{proof:uniqueness:crucial_lemma}. \\

\noindent Note that, as $j \to \infty$, $\psi_{j} \to 1$ in $\Omega$. Therefore, to conclude \eqref{res:global_energy_equality} from \eqref{res:local_energy_equality} we only have to establish:
$$
 \int_0^t \int_{\Omega} \alpha(s,x) \cdot \nabla\psi_{j}(x) \, T_k(u(s,x)) \,\diff x \diff s \to 0 \mbox{ as } j \to \infty.
$$
Since $\nabla\psi_{j}(x) = 0$ for $x \in \Omega_j$, we write for some constant $C_{\alpha}$ and $C_u$ to be chosen later:
\begin{equation}\label{proof_uniq:Youngsplit}
\begin{split}
\int_0^t \int_{\Omega \setminus \Omega_j} \alpha(s,x) \cdot \nabla\psi_{j}(x) \, T_k(u(s,x)) \,\diff x \diff s \leq 
C_{\alpha} C_u \int_0^t \int_{\Omega \setminus \Omega_j} M^*\left(s,x,\left|\frac{\alpha(s,x)}{C_{\alpha}}\right|\right) \,\diff x \diff s \, + \\   +\,
C_{\alpha} C_u \int_0^t \int_{\Omega \setminus \Omega_j} M\left(s,x,\left|\frac{\nabla\psi_{j}(x) \, T_k(u(s,x))}{C_{u}}\right|\right) \,\diff x \diff s,
\end{split}
\end{equation}
where we have applied Young's inequality (Lemma \ref{intro:Hold_Mus_inequality}). Since $\alpha \in L_{M*}(\Omega_T)$, there is $C_{\alpha}$ so that $M^*\left(s,x,\left|\frac{\alpha(s,x)}{C_{\alpha}}\right|\right) \,\diff x \diff s < \infty$. Choosing such $C_{\alpha}$, the first integral on the (RHS) of \eqref{proof_uniq:Youngsplit} tends to $0$ as $j \to \infty$ due to integrability of the integrand. Moreover, the second integral on the (RHS) of \eqref{proof_uniq:Youngsplit} converges to 0 due to Lemma \ref{proof:uniqueness:crucial_lemma}.
\end{proof}
\begin{rem}\label{rem:energy_equality}
Using Dominated Convergence Theorem, \eqref{res:global_energy_equality} implies that for a.e. $t \in (0,T)$:
\begin{equation}\label{res:global_energy_equality_without_truncation}
\frac{1}{2} \int_{\Omega} \left[u^2(t,x) - u_0^2(x) \right] \diff x =
- \int_0^t \int_{\Omega} \alpha(s,x) \cdot \nabla u(s,x) \, \diff x \diff s + \int_0^t \int_{\Omega} f(s,x) \,u(s,x) \, \diff x \diff s.
\end{equation}
\end{rem}
\begin{proof}[Proof of Theorem \ref{mainresult:uniqueness}] Energy equality \eqref{theorem:energyequality} follows from Remark \ref{rem:energy_equality}. Now, suppose there are two solutions $u$ and $v$ to \eqref{intro:parabolic_eq}. Then, their difference satisfies weak formultion for
$$
(u-v)_t = \DIV \left[A(t,x,\nabla u) - A(t,x,\nabla v) \right]
$$
with zero initial condition. Using \eqref{res:global_energy_equality_without_truncation} with 
$
\alpha(t,x) = A(t,x,\nabla u) - A(t,x,\nabla v),
$
we obtain for a.e. $t \in (0,T)$:
$$
\frac{1}{2}\int_{\Omega} (u(t,x) -v(t,x))^2 \diff x =
- \int_0^t \int_{\Omega} \Big[A(s,x,\nabla u) - A(s,x,\nabla v)\Big] \cdot \Big[\nabla u(s,x) - \nabla v(s,x)\Big] \, \diff x \diff s
$$
which due to weak monotonicity \ref{intro:assumA_mono} in Assumption \ref{intro:ass_on_A} implies $u = v$ a.e. in $\Omega_T$.
\end{proof}
{
\appendix
\section{Proofs concerning Examples \ref{ex:Nfunctions} and \ref{ex:operatorsA}}
 \subsection{Example \ref{ex:Nfunctions}}\label{app:example_Nfunctions}
Let $M(t,x,\xi) = |\xi|^{p(t,x)}$. We want to establish condition \ref{cond:ass_isotropic_case} in Remark \ref{rem:cond:ass_isotropic_case}. Fix $t \in (0,T)$ and $x,y \in \Omega$. When $|\xi| > 1$
$$
\frac{M(t,x,\xi)}{M(t,y,\xi)} = |\xi|^{p(t,x)-p(t,y)} \leq |\xi|^{|p(t,x)-p(t,y)|} \leq |\xi|^{-\frac{C}{\log|x-y|}}
$$
since $p(t,x) \in L^{\infty}(0,T; \Clog(\Omega))$. We let $\Theta(t,\delta,\xi):= |\xi|^{-\frac{C}{\log \delta}}$ so that for all $\widetilde{C}>1$ we have
$$
\Theta(t,\delta,\widetilde{C}\delta^{-1}) \leq \left(\widetilde{C} \, \delta^{-1}\right)^{-\frac{C}{\log \delta}} =\left(\delta/\widetilde{C} \right)^{\frac{C}{\log \delta}} = e^{\frac{C}{\log \delta}\, \log(\delta/\widetilde{C})} = e^C\, e^{-\frac{C\, \log\widetilde{C}}{\log \delta}}
$$
so that $\limsup_{\delta \to 0} \Theta(t,\delta,\widetilde{C}\delta^{-1})$ is bounded.\\

\noindent Similarly, let $M(t,x,\xi) = |\xi|^{p(t,x)} +a(t,x)\,|\xi|^{q(t,x)}$ and suppose that $q(t,y) - p(t,y) \leq \alpha$. Then, for $|\xi| > 1$ we have
\begin{align*}
\frac{M(t,x,\xi)}{M(t,y,\xi)} &= \frac{|\xi|^{p(t,x)} +a(t,x)\,|\xi|^{q(t,x)}}{|\xi|^{p(t,y)} +a(t,y)\,|\xi|^{q(t,y)}} = \frac{|\xi|^{q(t,x)}}{|\xi|^{q(t,y)}} \frac{|\xi|^{p(t,x)-q(t,x)} +a(t,x)}{|\xi|^{p(t,y)-q(t,y)} +a(t,y)} \leq \\
&\leq |\xi|^{q(t,x)-q(t,y)}\, \left[\frac{|\xi|^{p(t,x)-q(t,x)}}{|\xi|^{p(t,y)-q(t,y)}} + \frac{a(t,x) - a(t,y)}{|\xi|^{p(t,y)-q(t,y)}} + 1 \right] \\
&\leq |\xi|^{q(t,x)-q(t,y)}\, \left[{|\xi|^{p(t,x)-p(t,y)}}\,{|\xi|^{q(t,y)-q(t,x)}} + \frac{a(t,x) - a(t,y)}{|\xi|^{p(t,y)-q(t,y)}} + 1 \right]\\
&\leq |\xi|^{-\frac{C}{\log|x-y|}}\, \left[{|\xi|^{-\frac{C}{\log|x-y|}}}\,{|\xi|^{-\frac{C}{\log|x-y|}}} + {|a|_{\alpha}\,|x-y|^{\alpha}} \, {|\xi|^{q(t,y)-p(t,y)}} + 1 \right] \\
&\leq |\xi|^{-\frac{C}{\log|x-y|}}\, \left[{|\xi|^{-\frac{C}{\log|x-y|}}}\,{|\xi|^{-\frac{C}{\log|x-y|}}} + {|a|_{\alpha}\,|x-y|^{\alpha}}\,{|\xi|^{\alpha}} + 1 \right]
\end{align*}
where $|a|_{\alpha}$ is a constant such that $|a(t,x) - a(t,y)| \leq |a|_{\alpha}\,|x-y|^{\alpha}$. Hence, we define
$$
\Theta(t,\delta,\xi)=|\xi|^{-\frac{C}{\log \delta}}\, \left[{|\xi|^{-\frac{C}{\log\delta}}}\,{|\xi|^{-\frac{C}{\log \delta}}} + {|a|_{\alpha}\,\delta^{\alpha}}\,{|\xi|^{\alpha}} + 1 \right].
$$
We have already seen that $\left(\widetilde{C} \, \delta^{-1}\right)^{-\frac{C}{\log \delta}}$ is bounded when $\delta \to 0$. It follows that $\Theta(t,\delta,\widetilde{C}\delta^{-1})$ is bounded for such $\delta$.
}
\subsection{Example \ref{ex:operatorsA}}\label{app:example_A}
{ In both examples, the only nontrivial condition in Assumption \ref{intro:ass_on_A} is \ref{intro:ass_on_A:coercgr} (growth and coercivity). For \ref{exA:ptx}, we study ${A}(t,x,\xi) = |\xi|^{p(t,x)-2}\, \xi$ with $M(t,x,\xi) = |\xi|^{p(t,x)} = A(t,x,\xi)\cdot \xi$ so that we only need to verify
$$
M^*(t,x,A(t,x,\xi)) \leq C \, A(t,x,\xi) \cdot \xi
$$
for some numerical constant $C$. As we know that $M^*(t,x,\xi) \leq C\, |\xi|^{p'(t,x)}$ where $p'(t,x)$ is H\"older conjugate of $p(t,x)$ (i.e. $\frac{1}{p(t,x)} + \frac{1}{p'(t,x)} = 1$) we have
$$
M^*(t,x,A(t,x,\xi)) \leq C\, \left||\xi|^{p(t,x)-2}\, \xi \right|^{p'(t,x)} = C\,\left|\xi \right|^{(p(t,x)-1)\,p'(t,x)} = C\,\left|\xi \right|^{p(t,x)} = C \, A(t,x,\xi) \cdot \xi.
$$
For \ref{exA:double_phase}, we have $A(t,x,\xi)=|\xi|^{p(t,x)-2}\,\xi + a(t,x)\, |\xi|^{q(t,x)-2}\,\xi$ and $M(t,x,\xi) = |\xi|^{p(t,x)} + a(t,x)\, |\xi|^{q(t,x)}$. Again, since $A(t,x,\xi) \cdot \xi = M(t,x,\xi)$ we only have to prove
$$
M^*(t,x,A(t,x,\xi)) \leq C \, A(t,x,\xi) \cdot \xi
$$
for some numerical constant $C$. Using Definition \ref{intro:def_compl_func},
\begin{align*}
M^*(t,x,A(t,x,\xi)) &= \sup_{\eta \in \R^d} \left\{\eta \, A(t,x,\xi) - M(t,x,\xi) \right\}\\
&\leq \sup_{\eta \in \R^d} \left\{\eta \cdot \left( |\xi|^{p(t,x)-2}\,\xi + a(t,x)\, |\xi|^{q(t,x)-2}\,\xi\right) - \left(|\xi|^{p(t,x)} + a(t,x)\, |\xi|^{q(t,x)}\right) \right\} \\
&\leq \sup_{\eta \in \R^d} \left\{\eta \cdot \xi \, |\xi|^{p(t,x)-2}  - |\xi|^{p(t,x)} \right\} + a(t,x)\,
\sup_{\eta \in \R^d} \left\{\eta \cdot \xi \, |\xi|^{q(t,x)-2} - |\xi|^{q(t,x)} \right\}
\end{align*}
We introduce auxillary notation $M_1(t,x,\xi) = |\xi|^{p(t,x)}$, $A_1(t,x,\xi) = |\xi|^{p(t,x)-2}\,\xi$ as well as $M_2(t,x,\xi) = |\xi|^{q(t,x)}$, $A_2(t,x,\xi) = |\xi|^{q(t,x)-2}\,\xi$ and we recognize that
\begin{align*}
M^*(t,x,A(t,x,\xi)) &\leq M_1^*(t,x,A_1(t,x,\xi)) + a(t,x) \, M_2^*(t,x,A_2(t,x,\xi)) \\ &\leq A_1(t,x,\xi)\cdot\xi + a(t,x) \, A_2(t,x,\xi)\cdot\xi = \mathcal{A}(t,x,\xi) \cdot \xi
\end{align*}
which is justified by computations for the variable exponent case.
}
\bibliographystyle{abbrv}
\bibliography{parpde_mo_discmodul}
\end{document}